\documentclass[a4paper,11pt]{article}
\usepackage[a4paper, total={6.5in, 10in}]{geometry}
\usepackage{amsfonts,amsmath,amssymb,amsthm,enumerate,bm}
\usepackage{xcolor}

\allowdisplaybreaks[4] 
\theoremstyle{plain}

\newtheorem{theorem}{\noindent\bf Theorem}[section]

\newtheorem{lemma}[theorem]{\noindent\bf Lemma}

\theoremstyle{remark}
\newtheorem{remark}[theorem]{\noindent\bf Remark}
\theoremstyle{definition}

\numberwithin{equation}{section}

\def\be{\begin{eqnarray}}
\def\ee{\end{eqnarray}}
\def\ben{\begin{eqnarray*}}
\def\een{\end{eqnarray*}}
\def\benum{\begin{enumerate}}
\def\eenum{\end{enumerate}}
\newcommand{\lr}{\left(}
\newcommand{\rr}{\right)}

\newcommand{\lge}{\left\{ }
\newcommand{\rge}{\right\} }

\title{\bf \Large {Multiplication operators on amalgam of  $l^{q), \theta}$ and $L^{p}$}}
\author{ Monika Singh \footnote{The corresponding author.}  \hspace{2pt} and Jitendra Kumar} 

\date{}
\begin{document}
\maketitle
\begin{abstract}
\noindent

We define the grand amalgam Lebesgue function space $l^{q), \theta}(L^p),$ and study the fundamental structural properties of the space, including completeness. Then we define the small Lebesgue sequence space and study its function space properties. Furthermore, we prove a version of  the H\"{o}lders inequality on the frame work of these spaces. Finally, we study the multiplication operator in the setting of these spaces.

\bigskip

\end{abstract}

\section{Introduction}

Let $\mathbb N, ~\mathbb N_0,~ \mathbb Z,~ \mathbb R$ and $\mathbb C$ denote set of positive integers, non-negative integers, integers, real numbers and complex numbers, respectively. For $1 \leq p < \infty$ and  $1 \leq q < \infty$, the classical amalgam  space, denoted by $l^q (L^p),$ consist of all measurable functions $f: \mathbb R \rightarrow \mathbb C$ that are locally in $L^p$  and globally have $l^q$ behavior. That is the function $f \in L^p(I_n),$ for each $n \in \mathbb Z,$ where $I_n=[n,n+1),$  and the sequence of local $L^p$ norm lies in $l^q.$ Precisely, $ f \in l^q (L^p),$ if  
\be\label{kll}
\|f\|_{l^q (L^p)}:= \lr\displaystyle\sum_{n=-\infty}^\infty \lr\int_{I_n} |f(x)|^p dx\rr^{\frac{q}{p}}\rr^{\frac{1}{q}}=\big\|\| f \chi_{I_n}\| _{L^p}\big\|_{l^q}.
\ee
These spaces were first introduced by Norbert Wiener \cite{w1, w2} for specific values of $p$ and $q$ while a more systematic study was later carried out by F. Holland \cite{ho}. Further, for historical background and applications of these spaces in various contexts can be found in \cite{fs} and \cite{he}.

In 1992, Iwaniec and Sbordone \cite{is} introduced grand Lebesgue spaces in search of finding the minimal conditions for the integrability of the Jacobian. Let $1 < q < \infty$ and $I=(0,1),$ the grand Lebesgue space, $L^{q)}(I)$, is a space of all measurable functions such that 
\be\label{1.3}
 \|f\|_{L^{q)}}:= \sup_{0<\varepsilon < q-1} \varepsilon^{\frac{1}{q-\varepsilon}} \lr\int_0^1 |f(x)|^{q-\varepsilon} dx\rr^{\frac{1}{q-\varepsilon}} = \sup_{0<\varepsilon < q-1} \varepsilon^{\frac{1}{q-\varepsilon}} \|f\|_{L^{q-\varepsilon}}< \infty.
\ee
The terminology \emph{grand} arises from the following chain of embeddings:
\be\label{1.1}
L^q(I) \hookrightarrow L^{q)}(I) \hookrightarrow L^{q-\varepsilon}(I), ~~ 0<\varepsilon < q-1.   
\ee
The middle embedding in \eqref{1.1} is strict. For instance $x^{\frac{-1}{q}} \in L^{q)}(I),$ but it does not belong to the standard $L^q(I)$.  The structural properties of these spaces were studied by A. Fiorenza in \cite{fi}  and \cite{fk}. Currently, these spaces remain an active area of research and attract significant interest from many scholars. I, along with several others, have investigated these spaces from various perspectives. For further details, one may refer to \cite{jss2}, \cite{kmrs} and references there in. 

In 2018, the discrete version of the grand Lebesgue spaces is defined by Rafeiro, Samko and Umarkhadzhiev \cite{rsu}.  Let $X$ denote either $\mathbb N,~ \mathbb N_0$ or  $\mathbb Z.$ For $1\leq q< \infty$ and $\theta>0$, grand Lebesgue sequence space, $l^{q), \theta}$, consist of all real sequences $x=\lge x_n \rge_{n \in X},$   such that 
\be\label{kl}
 \|x\|_{l^{q),\theta}} := \displaystyle\sup_{\varepsilon>0 } \lr\varepsilon^\theta \sum_{n \in X} |x_n|^{q(1+\varepsilon)}\rr^{\frac{1}{q(1 + \varepsilon)}}=\displaystyle \sup_{ \varepsilon>0 } ~ \varepsilon^{\frac{\theta}{q(1 + \varepsilon)}} \|x\|_{l^{q{(1 + \varepsilon)}}} <\infty.
\ee
In \cite{rsu}, authors have shown that the range of $\varepsilon >0$ in the norm $\|x\|_{l^{q),\theta}}$ can be taken on the finite interval $(0,~ \varepsilon_0),$ for some $\varepsilon_0 \in (0, \infty ).$ Precisely, they have shown that  the norm $\|x\|_{l^{q),\theta}}$ is equivalent to the norm
\be\label{kmm}
\|x\|_{l^{q),\theta}, \varepsilon_0}:=\displaystyle \sup_{0< \varepsilon\le\varepsilon_0 } ~ \varepsilon^{\frac{\theta}{q(1 + \varepsilon)}} \|x\|_{l^{q{(1 + \varepsilon)}}}.
\ee
Also, embeddings \eqref{1.1} hold in the discrete case also.
\begin{equation} \label{1.2}
 l^{q(1-\sigma)} ~ \hookrightarrow ~l^q ~ \hookrightarrow~l^{q),\theta_1}~\hookrightarrow~ l^{q),\theta_2} ~\hookrightarrow ~l^{q(1+\delta)} 
\end{equation}
for $0 < \sigma < \frac{1}{q'}$, $\delta > 0$,$0 < \theta_1 \leq \theta_2,$ where $\frac{1}{q}+\frac{1}{q'}=1.$ Second embedding in \eqref{1.2} serves as the theoretical justification for the definition of \emph{grand} Lebesgue  sequence space. In the paper \cite{t1}, author has also introduced grand Lebesgue sequence space, where the norm is defined in a discrete form analogous to the norm given in \eqref{1.3}. For clarity, we denote this norm by  $\|.\|^1_{l^{q),\theta}}.$
\be
\nonumber \|x\|^1_{l^{q),\theta}}:=\displaystyle \sup_{0< \varepsilon ~\le q-1} ~ \varepsilon^{\frac{\theta}{q- \varepsilon}} \|x\|_{l^{q- \varepsilon}}.
\ee 
But with this definition of norm, the following embedding hold:
\be\label{1.5}
\nonumber l^1 \hookrightarrow l^{q)} \hookrightarrow l^q.
\ee
That is the grand Lebesgue sequence space is embedded in the Lebesgue sequence space. Also, this embedding is strict. Let $1 < \alpha < q,$ and $x=\lge(\frac{1}{n})^{\frac{\alpha}{q}}\rge_{n \in \mathbb N}.$ Clearly $ x\in l^q$ but $x \notin l^{q)}$ because

\begin{align*}
 \|x\|^1_{l^{q),\theta}} = &\displaystyle \sup_{0< \varepsilon ~\le q-1} ~ \varepsilon^{\frac{\theta}{q- \varepsilon}} \|x\|_{l^{q- \varepsilon}}\\
= & \displaystyle \sup_{0< \varepsilon ~\le q-1} ~ \varepsilon^{\frac{\theta}{q- \varepsilon}} \lr\sum_{n \in \mathbb N}\lr\lr\frac{1}{n}\rr^{\frac{\alpha}{q}}\rr^{q-\varepsilon}\rr^{\frac{1}{q-\varepsilon}}\\
\ge & \varepsilon_0^{\frac{\theta}{q- \varepsilon_0}} \lr\sum_{n \in \mathbb N}\lr\lr\frac{1}{n}\rr^{\frac{\alpha}{q}}\rr^{q-\varepsilon_0}\rr^{\frac{1}{q-\varepsilon_0}}=\infty,
\end{align*}
where $\varepsilon_0= q- \frac{q}{\alpha} < q-1.$ In \cite{t1},  the author amalgamated the grand $L^{p)}$ space with itself, termed the resulting space as the grand Wiener amalgam spaces, and studied the various space properties of these amalgamated spaces. In a subsequent work \cite{t2} the author introduced a discrete version of the previously defined spaces and further investigated some properties with in this new framework.

Many of the researchers have explored multiplication operator in the frame work of various function spaces viz., amalgam spaces with variable setting, grand Lorentz sequence space, mixed norm Lebesgue spaces etc, one may refer  \cite{ak}, \cite{o}, \cite{ch} and references therein.
In this paper, we also study the multiplication operator in the frame work of amalgam of $l^{q), \theta}$ and $L^{p}$ spaces. To work on the primary objective, we first establish the space  $l^{q), \theta}(L^p)$ in Section 2, and refer to it as the \emph{ Grand amalgam Lebesgue function space.} We begin by studying various inclusion properties of these spaces in the context of the parameters $p, q, \theta.$ Subsequently, we prove that  $l^{q), \theta}(L^p)$ is a  Banach space. To study the  version of the Hölder’s inequality within this framework, we need to define the  auxiliary space of the grand amalgam Lebesgue function space. In this direction, in Section 3, first we define the \emph{small Lebesgue sequence space} and study its various structural viz., Banach function space properties. Finally, in Section 4, we explore the boundedness of the multiplication operators on these spaces.

Throughout the paper, whenever we refer any of the spaces such as $l^q,$
$l^{q), \theta},$ $l^q(L^p)$ etc,  the index set of the sequences is to be taken $X,$ unless stated otherwise. In case of examples, we shall specify it explicitly. We denote  positive constants by the same symbol  $c,$ which may take different values from line to line within the same argument.

\section{A Banach space $l^{q), \theta}(L^p)$ and inclusion properties }

In this section we define the grand amalgam Lebesgue function spaces denoted by $l^{q), \theta}(L^p)$, and study their fundamental  properties. We establish that these spaces are Banach spaces. Furthermore, we explore the inclusion relation among the grand amalgam Lebesgue function spaces that arise from  variation in  the parameters $p, q, \theta.$ 

  Let  $1 \leq p < \infty,$  $1 \leq q < \infty$, and $\theta > 0$ then the \emph{grand amalgam Lebesgue function space} $l^{q), \theta}(L^p)$ is the collection of all complex valued measurable functions $g$  defined on $\displaystyle\bigcup_{k\in X} I_k,$ where $I_k=[k, k+1),$ with $ g\chi_{I_k} \in L^p,$ for each $ k \in X$ and  the sequence of local $L^p$-norm lies in $\displaystyle\bigcap_{\varepsilon>0}l^{q(1+\varepsilon)},$  such that
\begin{equation}\label{eq10}
 \|g\|_{p,q),\theta} := \sup_{\varepsilon > 0} \lr {\varepsilon^{\theta}}  \sum_{k \in  X} \left(\int_{k}^{k+1} |g(x)|^{p} \, dx \right)^ \frac{q(1+\varepsilon)}{p}\rr^\frac{1}{{q(1+\varepsilon)}} < \infty.
\end{equation}
It can also be written as 
\be
\nonumber\|g\|_{p,q),\theta}=\sup_{\varepsilon > 0}~ \varepsilon^{\frac{\theta}{q(1+\varepsilon)}} \big\| \| g \chi_{I_k}\|_{L^p} \big\|_{l^{q(1+\varepsilon)}}=\sup_{\varepsilon > 0}~ \varepsilon^{\frac{\theta}{q(1+\varepsilon)}} \| g \|_{l^{q(1+\varepsilon)}({L^p})},
\ee
as well as
\be\label{km}
\|g\|_{p,q),\theta}=\sup_{\varepsilon > 0}~ \varepsilon^{\frac{\theta}{q(1+\varepsilon)}} \big\| \| g \chi_{I_k}\|_{L^p} \big\|_{l^{q(1+\varepsilon)}}=\big\|\| g \chi_{I_k}\|_{L^p}\big\|_{l^{q), \theta}}.
\ee
by \eqref{kll} and \eqref{kl}.
\noindent We give an example of a measurable function which is in $l^{q), \theta}(L^p).$\\
{\bf Example 1.} 
Let  
$g(x) = \displaystyle\sum_{n \in \mathbb{N}} n^{-\frac{1}{q}}~ \lr\ln(n+1)\rr^{-a} \chi_{I_n}(x),$ where $ a \in \mathbb R.$ Consider,
\begin{align*}
\|g\|_{p,q),\theta} & = \sup_{\varepsilon > 0} \lr {\varepsilon^{\theta}}  \sum_{n \in \mathbb{N}} \left(\int_{n}^{n+1} |g(x)|^{p} \, dx
\right)^ \frac{q(1+\varepsilon)}{p}\rr^\frac{1}{{q(1+\varepsilon)}}\\
& =\sup_{\varepsilon > 0} \lr {\varepsilon^{\theta}}  \sum_{n \in \mathbb{N}} \lr n^{-\frac{1}{q}} \lr\ln(n+1)\rr^{-a} \rr^ {q(1+\varepsilon)}\rr^\frac{1}{{q(1+\varepsilon)}}\\
& = \|x\|_{l^{q), \theta}}, 
\end{align*}
where $x= \lge n^{-\frac{1}{q}} \lr\ln(n+1)\rr^{-a}\rge_{n \in \mathbb N}.$ By  Lemma 3.4 \cite{rsu}, $ \|x\|_{l^{q), \theta}}$ is finite if and only if 
$\frac{1 - \theta}{q} \leq a \leq \frac{1}{q}.$ Therefore, $g \in l^{q), \theta}(L^p)$ if and only if 
$\frac{1 - \theta}{q} \leq a \leq \frac{1}{q}.$ In particular, for $a=0,$ we have  $g(x)= \displaystyle\sum_{n \in \mathbb{N}} n^{-\frac{1}{q}}~\chi_{I_n}(x) \in l^{q), \theta}(L^p)$ if and only if $\theta \ge 1.$

 Now, we prove that grand amalgam Lebesgue function space is a Banach space. To begin with, we show that it is a normed space. This follows from the fact that both Lebesgue sequence space $l^q$ and Lebesgue function space $L^p$ are normed spaces for $q,~p \ge 1.$ However, for the sake of completion, we provide a proof below.

\begin{theorem} \label{thm 1}
 Let $1 \leq p < \infty$, $1 \leq q < \infty$ and $\theta >0$ then $l^{q), \theta}(L^p)$ is a normed space.
\end{theorem}
\begin{proof}Let $g, ~f \in l^{q), \theta}(L^p)$ and $ \lambda $ be a scalar. 
\begin{enumerate}[(a)]
\item \emph{Non-negativity:} As $\|g\chi_{I_k}\|_{L^p} \geq 0,$ for all $k \in  X,$ consequently we have $\big\| \| g \chi_{I_k}\|_{L^p} \big\|_{l^{q(1+\varepsilon)}} \ge 0$, for every $\varepsilon >0,$ hence $\|g\|_{p,q),\theta} \geq 0.$

\item Suppose $\|g\|_{p,q),\theta} = 0.$ I.e., $\displaystyle\sup_{\varepsilon > 0}~ \varepsilon^{\frac{\theta}{q(1+\varepsilon)}} \big\| \| g \chi_{I_k}\|_{L^p} \big\|_{l^{q(1+\varepsilon)}}=0.$ This implies that $\big\| \| g \chi_{I_k}\|_{L^p} \big\|_{l^{q(1+\varepsilon)}} = 0$, for every $\varepsilon >0.$ Since Lebesgue sequence space $l^{q(1+\varepsilon)}$ and $L^p$ space are normed spaces, we have $g=0$ almost everywhere. Converse can be obtained trivially. Hence, $\|g\|_{p,q),\theta} = 0 $ if and only if  $g = 0$ almost everywhere.

\item \emph{Homogeneity:} 
\begin{align*}
 \|\lambda g\|_{p,q),\theta} &= \displaystyle\sup_{\varepsilon > 0}~ \varepsilon^{\frac{\theta}{q(1+\varepsilon)}} \big\| \| \lambda g \chi_{I_k}\|_{L^p} \big\|_{l^{q(1+\varepsilon)}}
= \displaystyle\sup_{\varepsilon > 0}~ \varepsilon^{\frac{\theta}{q(1+\varepsilon)}} \big\| |\lambda| \| g \chi_{I_k}\|_{L^p} \big\|_{l^{q(1+\varepsilon)}}\\
&=   \displaystyle\sup_{\varepsilon > 0}~ \varepsilon^{\frac{\theta}{q(1+\varepsilon)}} |\lambda|\big\| \| g \chi_{I_k}\|_{L^p} \big\|_{l^{q(1+\varepsilon)}}= |\lambda|\| g\|_{p,q),\theta} 
\end{align*}

\item \emph{Triangle inequality:} 
Consider,
\begin{align*}
\| f + g \|_{p,q),\theta}  &= \displaystyle\sup_{\varepsilon > 0}~ \varepsilon^{\frac{\theta}{q(1+\varepsilon)}} \big\| \| (f+g) \chi_{I_k}\|_{L^p} \big\|_{l^{q(1+\varepsilon)}}\\
 &\le \displaystyle\sup_{\varepsilon > 0}~ \varepsilon^{\frac{\theta}{q(1+\varepsilon)}} \big\| \| f \chi_{I_k}\|_{L^p}+\| g \chi_{I_k}\|_{L^p} \big\|_{l^{q(1+\varepsilon)}}\\
&\le \displaystyle\sup_{\varepsilon > 0}~ \varepsilon^{\frac{\theta}{q(1+\varepsilon)}} \left( \big\| \| f \chi_{I_k}\|_{L^p} \big\|_{l^{q(1+\varepsilon)}}+\big\| \| g \chi_{I_k}\|_{L^p} \big\|_{l^{q(1+\varepsilon)}}\right) \\
& \leq \| f \|_{p,q),\theta} + \| g \|_{p,q),\theta}.
\end{align*}
\end{enumerate}
\end{proof}
To achieve the goal that $l^{q), \theta}(L^p)$ space is a Banach space, next we show that grand Lebesgue squence space $l^{q), \theta}$ is a Banach space.

\begin{remark}
In the paper \cite{o}, the author has defined the grand  Lorentz sequence space and proved its completeness. The completeness of the space \( l^{q), \theta }\) can also be obtained by applying similar arguments as used in  \cite{o} with the same choice of exponents. But here we provide an alternate proof for it.
\end{remark}

\begin{theorem}\label{th1}
For $\theta >0$ and $ 1 \leq q < \infty ,$ the space $ l^{q), \theta }$  is a Banach space.
\end{theorem}
\begin{proof}
Let \(\{x_n\}_{n \in X}\) be a Cauchy sequence in \( l^{q), \theta} .\) 
Let $ \delta > 0 $ be given.
Since \(\{x_n\}\) is a Cauchy sequence in \( l^{q), \theta },\) for given \( \delta > 0 \) there exists an integer \( M \) such that
\begin{equation}\label{9}
\|x_n - x_m\|_{l^{q),\theta}}  < \frac{\delta}{3}, \quad \text{for all}~ n, m  \ge M.
\end{equation}
I.e., 
\be
\nonumber \sup_{\varepsilon > 0} ~{\varepsilon}^\frac{\theta}{q(1+\varepsilon)} \|x_n - x_m\|_{l^{q(1+\varepsilon)}
} < \frac{\delta}{3}, \quad \text{for all}~ n, m  \ge M.
\ee
This implies that all for $\varepsilon>0,$ we have
\be 
 \nonumber {\varepsilon}^\frac{\theta}{q(1+\varepsilon)} \|x_n - x_m\|_{l^{q(1+\varepsilon)}
} < \frac{\delta}{3},\quad \text{for all}~ n, m  \ge M.
\ee
Therefore,  for all \( \varepsilon > 0, \)
$\{x_n\}$ is a Cauchy sequence in the space $l^{q{(1+\varepsilon)}}.$
Since $ l^{q(1+\varepsilon)} $ is a complete normed space, it follows that $\{x_n\}_{n \in X}$ converges in this space, say to  $x \in  l^{q(1+\varepsilon)}.$ 
Note that this limit $x$ will be independent of $\varepsilon>0$ because for  \(0< \varepsilon_1 < \varepsilon_2, \) we have $
l^{q(1+\varepsilon_1)} \subset l^{q(1+\varepsilon_2)}.$

\noindent Now, we need to show that \( x \in l^{q), \theta} \) and \(\|x_n - x\|_{l^{q),\theta}} \to 0 \) as \( n \to \infty. \)\\
For $n \ge M, $ consider 
\[
\|x_n - x\|_{l^{q),\theta}}  = \sup_{\varepsilon > 0} {\varepsilon}^\frac{\theta}{q(1+\varepsilon)} \|x_n - x\|_{l^{q(1+\varepsilon)}
}.\] 
By the definition of supremum, for $\frac{\delta}{3}>0$  there exists \( \varepsilon_0 > 0, \) depending upon $n,$ such that:

\begin{equation}\label{10}
 \|x_n - x\|_{l^{q),\theta}}  -\frac{\delta}{3}\leq\varepsilon_{0}^\frac{\theta}{q(1+\varepsilon_{0})} \|x_n - x\|_{l^{q(1+\varepsilon_{0})}}.
\end{equation}
Since $\{x_m\}$ converges to $x$ in \( l^{q(1 + \varepsilon_0)}, \) for given \( \frac{\delta}{3}>0 \) there exists an integer \( M_1 \) such that

\begin{equation}\label{11}
\varepsilon_{0}^\frac{\theta}{q(1+\varepsilon_{0})} \|x_m - x\|_{l^{q(1+\varepsilon_{0})}}  < \frac{\delta}{3}, \quad \text{for all}~ m \geq M_1.
\end{equation}
Choose \( K > \max \{M, M_1\} \), then for $n \ge K,$ by using the triangle inequality in $l^{q),\theta}$, \eqref{10}, \eqref{9} and \eqref{11}, we get
\begin{align}\label{12}
\nonumber\|x_n - x\|_{l^{q),\theta}} & \leq \|x_n - x_K\|_{l^{q),\theta}}  + \|x_K - x\|_{l^{q),\theta}} \\\nonumber
& \le \|x_n - x_K\|_{l^{q),\theta}}+ \varepsilon_{0}^\frac{\theta}{q(1+\varepsilon_{0})} \|x_K - x\|_{l^{q(1+\varepsilon_{0})}}+\frac{ \delta}{3}\\
& < \delta, \quad \text{for all} ~n \geq K.
\end{align}
Thus, the Cauchy sequence $\{x_n\}_{ n \in X}$ converges to $x$ in \( l^{q), \theta}. \)\\
Now by using \eqref{12}, the fact that $\{x_n\}_{ n \in X} \subset l^{q),\theta},$   and   
$
\|x\|_{l^{q),\theta}} \leq \|x_n - x\|_{l^{q),\theta}}   + \|x_n\|_{l^{q),\theta}},
$ we have  $\|x\|_{l^{q),\theta}} < \infty.$ Therefore, $ x \in l^{q), \theta}.$ Hence proved.
\end{proof}
 So, finally we prove the completeness of the grand amalgam Lebesgue function space.

\begin{theorem}\label{th2}
Let $\theta>0$ and  $1 \leq p,~q < \infty$ then the grand amalgam Lebesgue function space $l^{q),\theta}(L^p)$ is a Banach space.
\end{theorem}
\begin{proof}
To prove that $l^{q), \theta}(L^p)$  is a Banach space, we shall prove the  Riesz-Fischer property i.e.,  if \(\{g_k\}_{ k \in X}\) is a sequence of functions in  $l^{q), \theta}(L^p)$  with
$\displaystyle\sum_{ k \in X} \|g_k\|_{p,q),\theta} < \infty$ then $ \displaystyle\sum_{ k \in X} g_k $converges in $l^{q),\theta}(L^p).$
By using \eqref{km}, we can write 
\be
\nonumber \left\| \sum_{ k \in X} g_k \right\|_{p,q),\theta}=\sup_{\varepsilon > 0}~ \varepsilon^{\frac{\theta}{q(1+\varepsilon)}} \left\|\left\| \lr\sum_{ k \in X} g_k\rr \chi_{I_n} \right\|_{L^p}\right\|_{l^{q(1+\varepsilon)}}=\left\| \left\| \sum_{ k \in X} g_k \chi_{I_n} \right\|_{L^p} \right\|_{l^{q),\theta}}.
\ee
Now, by using Theorem \ref{th1} and the fact that for  $p \ge 1,$  \( L^p \) space is a Banach space, we can use the Riesz-Fischer property in both the spaces. Consequently, we obtain
\begin{align*}
\left\| \sum_{ k \in X} g_k \right\|_{p,q),\theta}&=\left\| \left\| \sum_{ k \in X} g_k \chi_{I_n} \right\|_{L^p} \right\|_{l^{q),\theta}}\\
& \le \left\| \sum_{ k \in X}\| g_k\chi_{I_n} \|_{L^p} \right\|_{l^{q),\theta}}\\
& \le \sum_{ k \in X}\big\| \| g_k \chi_{I_n} \|_{L^p} \big\|_{l^{q),\theta}}=\sum_{ k \in X}\| g_k \|_{p,q),\theta} < \infty. 
\end{align*}
Hence by the Riesz-Fischer property, we have proved that $l^{q),\theta}(L^p)$ is a Banach space. 
\end{proof}

\begin{remark}
By the completeness of the  classical Lebesgue space $L^p$ on $I_k, ~~ k \in X,$  together with  Theorem \ref{th1}, which gives the completeness of grand Lebesgue sequence space $l^{q),\theta}$, Theorem \ref{th2} can also be derived by applying Theorem 1 of \cite{ho}.
\end{remark}

\begin{remark}\label{r1}
In the definition \eqref{eq10} of the grand amalgam Lebesgue function space $l^{q),\theta}(L^p),$ the scaling function$\psi(\varepsilon)=\varepsilon^{\frac{1}{(1 + \varepsilon)}},$ for $\varepsilon >0, $  takes its maximum at  $\varepsilon=\frac{1}{W(\frac{1}{\mathrm e})} \approx 3.59$ with maximum value $\frac{1}{\mathrm{e} W(1 / \mathrm{e})} \approx  1.32.,$ where $W$ is the \emph{Lambert function} which is the inverse of the function $\phi: \mathbb R^+ \rightarrow \mathbb R^+,$ defined as $\phi(y)= y\mathrm e^y,$ see \cite{cghjk}.
\end{remark}

 Below we present the equivalent norm for the space $l^{q),\theta}(L^p),$ similar to \eqref{kmm}, in which the supremum is taken over the finite interval $0< \varepsilon \le \varepsilon_0,$ for some $\varepsilon>0.$ The  equivalence can be deduced from the proof of the Lemma 3.5, \cite{rsu}. For clarity we are giving proof here.
  
\begin{lemma}
The norm $\|g\|_{p,q),\theta}$ is equivalent to the norm
\be 
\nonumber \|g\|_{p,q),\theta,\varepsilon_0} := \sup_{0 < \varepsilon \le \varepsilon_0} \varepsilon^{\frac{\theta}{q(1 + \varepsilon)}} \|g\|_{p,q (1 + \varepsilon)} , \quad \varepsilon_0 \in (0, \infty).
\ee  
Moreover,
\be\label{115}  
\|g\|_{p,q),\theta,\varepsilon_0} \leq \|g\|_{p,q),\theta} \leq (c(\varepsilon_0))^{\frac{\theta}{q}} \|g\|_{p,q),\theta,\varepsilon_0},
\ee
 where $c\left(\varepsilon_0\right)=\frac{1}{\mathrm{e} W(1 / \mathrm{e})} \varepsilon_0^{-\frac{1}{1+\varepsilon_0}}.$ 
\end{lemma}
\begin{proof} Take some $\varepsilon_0 >0.$ Clearly by the definition of the supremum, we have 
\be\label{eq13}
\|g\|_{p,q),\theta,\varepsilon_0} \leq \|g\|_{p,q),\theta}.
\ee
Now, for $\varepsilon_0 < \varepsilon < \infty, $ $l^{q(1+\varepsilon_0)} \subset l^{q(1+\varepsilon)},$ we have $\big\| \| g \chi_{I_k}\|_{L^p} \big\|_{l^{q(1+\varepsilon)}} \le \big\| \| g \chi_{I_k}\|_{L^p} \big\|_{l^{q(1+\varepsilon_0)}}.$\\
Consider,
\begin{align}\label{eq14}
\nonumber\|g\|_{p,q),\theta}=& \sup_{\varepsilon > 0}~ \varepsilon^{\frac{\theta}{q(1+\varepsilon)}} \big\| \| g \chi_{I_k}\|_{L^p} \big\|_{l^{q(1+\varepsilon)}}\\\nonumber
=& \max\lge{\sup_{0<\varepsilon\le\varepsilon_0}~ \varepsilon^{\frac{\theta}{q(1+\varepsilon)}} \big\| \| g \chi_{I_k}\|_{L^p} \big\|_{l^{q(1+\varepsilon)}}}, \sup_{\varepsilon_0 <\varepsilon<\infty}~ \varepsilon^{\frac{\theta}{q(1+\varepsilon)}} \big\| \| g \chi_{I_k}\|_{L^p} \big\|_{l^{q(1+\varepsilon)}}\rge\\\nonumber
\le & \max\lge{\sup_{0<\varepsilon\le\varepsilon_0}~ \varepsilon^{\frac{\theta}{q(1+\varepsilon)}} \big\| \| g \chi_{I_k}\|_{L^p} \big\|_{l^{q(1+\varepsilon)}}}, \sup_{\varepsilon_0 <\varepsilon<\infty}~ \varepsilon^{\frac{\theta}{q(1+\varepsilon)}} \big\| \| g \chi_{I_k}\|_{L^p} \big\|_{l^{q(1+\varepsilon_0)}}\rge\\\nonumber
=& \|g\|_{p,q),\theta,\varepsilon_0}~ \max\lge 1, \sup_{\varepsilon_0 <\varepsilon<\infty} \lr\frac{\varepsilon^{\frac{1}{1+\varepsilon}}}{\varepsilon_0^{\frac{1}{1+\varepsilon_0)}}}\rr^{\frac{\theta}{q}}\rge\\
\le&  \|g\|_{p,q),\theta,\varepsilon_0}~c(\varepsilon_0)^{\frac{\theta}{q}},
\end{align}
where $c\left(\varepsilon_0\right)=\frac{1}{\mathrm{e} W(1 / \mathrm{e})} \varepsilon_0^{-\frac{1}{1+\varepsilon_0}}\ge 1$ and $\frac{1}{\mathrm{e} W(1 / \mathrm{e})} \approx  1.32.$
Hence, by \eqref{eq13} and \eqref{eq14}, we obtain the required equivalence \eqref{115}. 
\end{proof}

 In the next theorem we prove the following embeddings: 

\begin{theorem}\label{th4}
The following embeddings hold true:
\begin{enumerate}[(a)]
\item  If $1 \leq q_1 < q_2 < \infty$ and $1 \leq p < \infty,$ then 
 $l^{q_1), \frac{q_1}{q_2}\theta}(L^p) \hookrightarrow l^{q_2), \theta}(L^p).$

\item If $0 <\theta_1 \leq \theta_2$, 
then $ l^{q),\theta_1}(L^p) \hookrightarrow l^{q),\theta_2}(L^p).$

\item If $0<\theta_1 \leq \theta_2,$  $1 \leq p_2 < p_1 < \infty$ and $1 \leq q < \infty,$ then
$ l^{q),\theta_1}(L^{p_1}) \hookrightarrow l^{q),\theta_2}(L^{p_2}). $

\item For $1 \leq p < \infty,$ $1 \leq q < \infty$ and $\theta > 0$ then $l^q\left(L^p\right) \hookrightarrow l^{q), \theta}\left(L^p\right).$ I.e., classical amalgam space is contained in the grand amalgam Lebesgue function space.

\item If $\delta > 0$ and $0 < \sigma < \frac{1}{q'},$ where $q'$ is the conjugate of $q,$ then
$$ l^{q(1-\sigma)}(L^p) \hookrightarrow l^{q),\theta}(L^p) \hookrightarrow l^{q(1+\delta)}(L^p). $$
\end{enumerate}
\end{theorem}
\begin{proof}
\begin{enumerate}[(a)]
\item Let $g \in  l^{q_1), \frac{q_1}{q_2}\theta}(L^p)$. Since $q_1 < q_2$, for every $\varepsilon>0$ we have $l^{q_1(1+\varepsilon)} \subset l^{q_2(1+\varepsilon)}.$ I.e.,
$
\big\| \| g \chi_{I_k}\|_{L^p} \big\|_{l^{q_2(1+\varepsilon)}} \leq \big\|\| g \chi_{I_k}\|_{L^p} \big\|_{l^{q_1(1+\varepsilon)}}
.$ By using this fact, we get
\begin{align*}
\|g\|_{p,q_2),\theta}=&\sup_{\varepsilon > 0} \varepsilon^{\frac{\theta}{{q_2}(1+\varepsilon)}} \big\|\| g \chi_{I_k}\|_{L^p} \big\|_{l^{q_2(1+\varepsilon)}}
\le \sup_{\varepsilon > 0} \varepsilon^{\frac{\theta}{{q_2}(1+\varepsilon)}} \big\|\| g \chi_{I_k}\|_{L^p} \big\|_{l^{q_1(1+\varepsilon)}}\\
=&\sup_{\varepsilon > 0} \varepsilon^{\frac{\frac{q_1}{q_2}\theta}{q_1(1+\varepsilon)}} \big\|\| g \chi_{I_k}\|_{L^p} \big\|_{l^{q_1(1+\varepsilon)}}=
\|g\|_{p,q_1),\frac{q_1}{q_2}\theta}<\infty.
\end{align*}
Thus 
$
g \in l^{q_2),\theta}(L^p).
$\\
But this is the strict inclusion. Take $g(x) = \displaystyle\sum_{n \in \mathbb{N}} n^{-\frac{1}{q_2}} \chi_{I_n}(x).$ For $\theta \ge1,$ by Example 1, $
g \in l^{q_2),\theta}(L^p).$
Now, consider, 
\begin{align*}
\|g\|_{p,q_1),\frac{q_1}{q_2}\theta} =&\sup_{\varepsilon > 0}\varepsilon^\frac{{\frac{q_1}{q_2}\theta}}{{q_1(1+\varepsilon)}}    \lr \sum_{n \in \mathbb{N}} \left(\int_{n}^{n+1} |g(x)|^{p} \, dx
\right)^ \frac{q_1(1+\varepsilon)}{p}\rr^\frac{1}{{q_1(1+\varepsilon)}}\\
=&\sup_{\varepsilon > 0} \varepsilon^\frac{{\frac{q_1}{q_2}\theta}}{{q_1(1+\varepsilon)}}   \lr \sum_{n \in \mathbb{N}} \ n^{-\frac{q_1(1+\varepsilon)}{q_2}} \rr^\frac{1}{{q_1(1+\varepsilon)}}\\
=& \sup_{\varepsilon > 0} \varepsilon^\frac{{\frac{q_1}{q_2}\theta}}{{q_1(1+\varepsilon)}} \left\| n^{-\frac{1}{q_2}} \right\|_{l^{q_1(1+\varepsilon)}}.
\end{align*}
For $0 < \varepsilon \leq \alpha$, where $\alpha=\frac{q_2}{q_1} - 1,$ the series $x=\lge n^{-\frac{1}{q_2}}\rge_{n \in \mathbb N}$ diverges in the space $l^{q_1(1+\varepsilon)}$ because $\frac{q_1(1+\varepsilon)}{q_2} \le 1.$ 
 Therefore, $\|g\|_{p,q_1),\frac{q_1}{q_2}\theta} = \infty.$ Thus $g \notin l^{q_1),\frac{q_1}{q_2}\theta}(L^p)$.

\item Let $0<\theta_1 \leq \theta_2$, we need to prove 
$
l^{q),\theta_1}(L^p) \hookrightarrow l^{q),\theta_2}(L^p).
$
Let $g \in l^{q),\theta_1}(L^p)$, by using embeddings  \eqref{1.2} in grand Lebesgue sequence spaces, we get
\begin{align*}
\|g\|_{p,q),\theta_2}=&\big\|\| g \chi_{I_k}\|_{L^p}\big\|_{l^{q), \theta_2}}
\le c \big\|\| g \chi_{I_k}\|_{L^p}\big\|_{l^{q), \theta_1}}
=c\|g\|_{p,q),\theta_1}<\infty,
\end{align*}
where $c$ is a positive constant. Hence $g \in l^{q),\theta_2}(L^p)$.\\
Now, we show that the inclusion is strict.
Take $\theta_1=\frac{1}{2}$ and $\theta_2=2.$  Let $g(x) = \sum_{n \in \mathbb{N}} n^{-\frac{1}{q}} \chi_{I_n}(x).$
 By Example 1, 
$g \in l^{q), 2}\left(L^p\right) 
$ but
 $g \notin l^{q), \frac{1}{2}}\left(L^p\right)$ because$\theta_2=2>1$ and $\frac{1}{2}=\theta_1<1$.

\item Let $1 \leq p_2 < p_1 < \infty,$  $0<\theta_1 \leq \theta_2$ and $1 \leq q < \infty.$
Let $g \in l^{q),\theta_1}(L^{p_1}).$
Since $p_2 < p_1$ and Lebesgue measure of $I_k$ is 1, for each  $ k \in X,$ by the nesting property of the $L^p$ spaces we have
$ \| g \chi_{I_k}\|_{L^{p_2}} \leq \| g \chi_{I_k}\|_{L^{p_1}}.$ Also, for $x=\lge x_k\rge_{ k \in X}, ~ y=\lge y_k\rge_{ k \in X} \in l^{q(1+\varepsilon)}$ with $x \le y,$ i.e., $x_k \le y_k,$ for all $k \in X,$  by the lattice property of the Lebesgue sequence spaces we have $\|x\|_{l^{q(1+\varepsilon)}} \leq \|y\|_{l^{q(1+\varepsilon)}},$ for every $\varepsilon >0.$
By using these arguments and  embeddings  \eqref{1.2} in grand Lebesgue sequence spaces, we obtain,
\begin{align*}
\|g\|_{p_2,q),\theta_2} &=\sup_{\varepsilon > 0} \varepsilon^{\frac{\theta_2}{{q}(1+\varepsilon)}} \big\|\| g \chi_{I_k}\|_{L^{p_2}} \big\|_{l^{q(1+\varepsilon)}}\\
&=\big\|\| g \chi_{I_k}\|_{L^{p_2}}\big\|_{l^{q), \theta_2}}\\
&\le c \,\big\|\| g \chi_{I_k}\|_{L^{p_2}}\big\|_{l^{q), \theta_1}}
\\
&= c\,\sup_{\varepsilon > 0} \varepsilon^{\frac{\theta_1}{{q}(1+\varepsilon)}} \big\|\| g \chi_{I_k}\|_{L^{p_2}} \big\|_{l^{q(1+\varepsilon)}}\\
&\le c\,\sup_{\varepsilon > 0} \varepsilon^{\frac{\theta_1}{{q}(1+\varepsilon)}} \big\|\| g \chi_{I_k}\|_{L^{p_1}} \big\|_{l^{q(1+\varepsilon)}}<\infty,
\end{align*}
where $c$ is a positive constant. Hence
$
g \in l^{q),\theta_2}(L^{p_2}).$\\
Again this inclusion is strict. Let $g(x) = \sum_{n \in \mathbb{N}} n^{\frac{-1}{q}} \chi_{[n,n+1)}(x).$ Take $\alpha = \frac{p_2}{q} + p_1,$ $\theta_2 = (\alpha - p_2)\frac{q}{p_2}$ and  $\beta=p_1 - p_2  > 0.$ Since $\beta >0,$ we have $\theta_2>1.$ Therefore, by Example 1
$
 g \in l^{q),\theta_2}(L^{p_2}).
$
Now take $\theta_1 = (\alpha - p_1)\frac{q}{p_1}.$ Clearly $\theta_1 < \theta _2.$ Since  $\theta_1 = (\alpha - p_1)\frac{q}{p_1}=\frac{p_2}{p_1} < 1,$ by Example 1, $ g \notin l^{q),\theta_1}(L^{p_1}).$

\item Let $1 \leq p < \infty$, $1 \leq q < \infty$ and $\theta > 0.$
Let $g \in l^q\left(L^p\right),$ in view of the Remark \ref{r1} and by applying the nested property of the classical Lebesgue sequence spaces we get

\begin{align*}
\|g\|_{p, q), \theta}&=\sup_{\varepsilon > 0} {\varepsilon^\frac{\theta}{q(1+\varepsilon)}} \big\| \| g \chi_{I_k}\|_{L^p} \big\|_{l^{q(1+\varepsilon)}}\\ 
&\le \sup_{\varepsilon > 0} {\varepsilon^\frac{\theta}{q(1+\varepsilon)}} \big\| \| g \chi_{I_k}\|_{L^p}\big\|_{l^{q}} \\
&=
  \| g \|_{l^q(L^p)} \sup_{\varepsilon > 0} {\varepsilon^\frac{\theta}{q(1+\varepsilon)}} \\
	&=  \| g \|_{l^q(L^p)} \lr\frac{1}{\mathrm e{W(\frac{1}{\mathrm e})}}\rr^{\frac{\theta}{q}}<\infty.
\end{align*}
Therefore $ g \in  l^{q),\theta}(L^p).$\\
This inclusion is also strict because for $\theta \geq 1,$
$g(x) = \displaystyle\sum_{n \in \mathbb{N}} n^{\frac{-1}{q}} \chi_{[n,n+1)}(x)
 \in l^{q),\theta}(L^p)$  but $g \notin l^q(L^p)$ because
\be
\nonumber \|g\|_{l^q(L^p)}=\big\| \| g \chi_{I_n}\|_{L^p} \big\|_{l^q}=\lr
\sum_{ n \in \mathbb N} \lr\int_n^{n+1} \lr n^{\frac{-1}{q}} \chi_{I_n}\rr^p\rr^{\frac{q}{p}}\rr^\frac{1}{q}=\lr\sum_{n \in \mathbb N}\frac1
{n}\rr^\frac{1}{q}=\infty.
\ee

\item Let $\delta > 0.$
By applying embedding \eqref{1.2} we get
\begin{align}\label{eq16}
\|g\|_{l^{q(1+\delta)}(L^p)}&= \big\| \| g \chi_{I_k}\|_{L^p} \big\|_{l^{q(1+\delta)}} 
\le c\sup_{\varepsilon > 0} {\varepsilon^\frac{\theta}{q(1+\varepsilon)}} \big\| \| g \chi_{I_k}\|_{L^p}\big\|_{l^{q(1+\varepsilon)}}=c \|g\|_{p,q), \theta}<\infty,
\end{align}
where $c$ is a positive constant. \\
Now let $0< \sigma < \frac{1}{q'}.$ Again by applying embedding \eqref{1.2} we obtain
\begin{align}\label{eq17}
\|g\|_{p,q), \theta}
=\big\|\| g \chi_{I_k}\|_{L^p}\big\|_{l^{q),\theta}}
\le C \big\| \| g \chi_{I_k}\|_{L^p}\big\|_{l^{q(1-\sigma)}}
\le C \|g\|_{l^{q(1-\sigma)}(L^p)}<\infty,   
\end{align}
where $C$ is a positive constant. 
After combining \eqref{eq16} and \eqref{eq17}, we obtain
\[\|g\|_{l^{q(1+\delta)}(L^p)} \le c \|g\|_{p,q), \theta}\le c \,C \|g\|_{l^{q(1-\sigma)}(L^p)}.\]
Hence done.
\end{enumerate}
\end{proof}

\begin{remark}
Suppose $\theta>0$ is given. When $1\le p = q$, we know that the classical amalgam Lebesgue space coincides with Lebesgue  space, see \cite{ho}, i.e.,
$l^q(L^p) = L^p.$ By using  Theorem \ref{th4} (d) part  we get the following embedding:
\[ l^p(L^p)=L^p \hookrightarrow l^{p),\theta}(L^p). \]
I.e., $L^p$ space is strictly contained in grand amalgam Lebesgue function space $l^{p),\theta}(L^p).$
\end{remark}

We now prove a lemma which is required for one of the main theorems in Section 4.

\begin{lemma}\label{lm 1}
Let  $\{ f_n \}_{n \in \mathbb{N}}$ be a sequence in $l^{q), \theta}(L^p)$ and $\{ f_n  \}_{n \in \mathbb{N}} \to f $ in the norm of \(l^{q), \theta}(L^p)\), where \(f \in l^{q), \theta}(L^p)\). Then $\{ f_n \}_{n \in \mathbb{N}}$ has a subsequence which converges pointwise almost everywhere to f.

\begin{proof}
Suppose $\|f_n-f\|_{p,q),\theta} \to 0.$
Let $\eta>0$ be given then there exists a positive integer $K$ such that
\be
\nonumber
\|f_n-f\|_{p,q),\theta}=\sup_{\varepsilon > 0}~ \varepsilon^{\frac{\theta}{q(1+\varepsilon)}} \| f_n-f \|_{l^{q(1+\varepsilon)}({L^p})} <{\eta}, \,\,\text{for all}\,\, n \ge K.
\ee
This implies
\be
\nonumber \| f_n-f \|_{l^{q(1+\varepsilon)}({L^p})} <{\eta}, \,\,
\text{for all} \,\, n \ge K, \, \,\varepsilon>0.
\ee
So, for $0<\varepsilon \leq \varepsilon_0$,
\be
\nonumber\| f_n-f \|_{l^{q(1+\varepsilon_0)}({L^p})} \leq\| f_n-f \|_{l^{q(1+\varepsilon)}({L^p})} <{\eta}, ~~ \,\,
\text{for all} \,\, n \ge K.
\ee
Therefore,
 $\{ f_n  \}_{n \in \mathbb{N}} \to f$  in the norm of \(l^{q(1+\varepsilon_0)}(L^p)\). So, by using the Lemma 2.12 \cite{ay},
 $\{ f_n \}_{n \in \mathbb{N}}$ has a subsequence which converges pointwise almost everywhere to f.
\end{proof}
\end{lemma}

Now we define a subset of $l^{q), \theta}(L^p),$ which form a closed subspace of the grand amalgam Lebesgue function space.
Let $1 \leq p, q < \infty$ and $\theta > 0$. We define  the \emph{Vanishing grand amalgam Lebesgue function space,} and denote it by 
$l^{\overset{\circ}q), \theta}(L^p),$
\[
l^{\overset{\circ}q), \theta}(L^p)= \left\{ g \in l^{q), \theta}(L^p) \;\middle|\; 
\displaystyle\lim_{\varepsilon \to 0} \varepsilon^{\theta} \sum_{n \in X} \left( \int_n^{n+1} |g(x)|^p \, dx \right)^{\frac{q(1+\varepsilon)}{p}} = 0 \right\}.
\]
\begin{theorem}
$l^{\overset{\circ}q), \theta}(L^p)$ is a closed subspace of $l^{q), \theta}(L^p)$.
\end{theorem}
\begin{proof}
Let $g $ be a limit point of $l^{\overset{\circ}q), \theta}(L^p)$. This implies that
there exists a sequence
$\{ g_i \}_{i \in \mathbb{N}} \in l^{\overset{\circ}q), \theta}(L^p)$ such that $\{ g_i \}_{i \in \mathbb{N}} \to g$ in the norm of $l^{q), \theta}(L^p).$ 
I.e., for given $\eta>0$ there exists a positive integer $K$ such that
\be\label{301}
\|g_i-g\|_{l^{q), \theta}(L^p)}=\sup_{\varepsilon > 0} \left\{ {\varepsilon^{\theta}}  \sum_{n \in X} \left(\int_{n}^{n+1} |(g_i-g)(x)|^{p} \, dx \right)^ \frac{q(1+\varepsilon)}{p}\right\}^\frac{1}{{q(1+\varepsilon)}} <\frac{\eta}{2}, \,\,\text{for all}\,\, i \ge K.
\ee
Since $g_i \in l^{\overset{\circ}q), \theta}(L^p),$ for $i \ge K$ and $\frac{\eta}{2}>0,$ there exists $ \delta>0$  such that
\be\label{31}
 \varepsilon^\theta \sum_{n \in X} \left( \int_n^{n+1} |g_i(x)|^p \, dx \right)^{\frac{q(1+\varepsilon)}{p}}  < \left( \frac{\eta}{2} \right)^{q(1+\varepsilon)},
\quad \text{whenever } 0 < \varepsilon < \delta.
\ee
For $0 < \varepsilon < \delta,$ and $i \ge K,$  using \eqref{301}, \eqref{31} and triangle inequality, we obtain
\begin{align*}
\varepsilon^{\frac{\theta}{q(1+\varepsilon)}} 
\left( \sum_{n \in X} 
\left( \int_n^{n+1} |g(x)|^p \, dx \right)^{\frac{q(1+\varepsilon)}{p}} 
\right)^{\frac{1}{q(1+\varepsilon)}} &\leq
\sup_{0 < \varepsilon < \delta}
\left\{
\varepsilon^{\frac{\theta}{q(1+\varepsilon)}} 
\left( \sum_{n \in X} 
\left( \int_n^{n+1} |g(x)|^p \, dx \right)^{\frac{q(1+\varepsilon)}{p}} 
\right)^{\frac{1}{q(1+\varepsilon)}}
\right\}\\
&=\sup_{0 < \varepsilon < \delta} \varepsilon^{\frac{\theta}{q(1+\varepsilon)}} \| g\|_{l^{q(1+\varepsilon)}({L^p})}\\ 
&\le \sup_{0 < \varepsilon < \delta} \varepsilon^{\frac{\theta}{q(1+\varepsilon)}} \| g_i-g\|_{l^{q(1+\varepsilon)}({L^p})}+\sup_{0 < \varepsilon < \delta} \varepsilon^{\frac{\theta}{q(1+\varepsilon)}} \| g_i \|_{l^{q(1+\varepsilon)}({L^p})}\\
&\le \frac{\eta}{2}+ \frac{\eta}{2}=\eta.
\end{align*}
I.e., 
$$
\lim_{\varepsilon \to 0} \varepsilon^{\theta} \sum_{n \in X} \left( \int_n^{n+1} |g(x)|^p \, dx \right)^{\frac{q(1+\varepsilon)}{p}} = 0.
$$ This implies that $g \in  l^{\overset{\circ}q), \theta}(L^p).$  Hence $l^{\overset{\circ}q), \theta}(L^p)$ is a closed subspace of $l^{q), \theta}(L^p).$
\end{proof}

\section{H\"{o}lder's inequality in $l^{q), \theta}(L^p)$}
 In this section, we aim to prove the H\"{o}lder's inequality within the framework of the grand amalgam Lebesgue function spaces. To do so, we first prove H\"{o}lder's inequality  in the grand sequence space $l^{q),\theta}.$\\
 We begin by stating a lemma of independent interest.
\begin{lemma}
Let $\theta > 0$ is given and  $1 \leq p_i, q_i < \infty, ~~i = 1,2,3.$ Suppose that there exist positive constants $c$ and $C,$ respectively, such that
\be\label{eq18}
\|xy\|_{l^{q_3),\theta}} \leq c \|x\|_{l^{q_1),\theta}} \,\|y\|_{l^{q_2),\theta}}, \,\,\,\text{for}\,\,\, x \in l^{q_1),\theta},\, y \in l^{q_2),\theta},
\ee
where  $x=\lge x_k\rge_{k \in X}$ and $y=\lge y_k\rge_{k \in X},$
and
\be\label{eq19}
\|(hs)\chi_{E}\|_{L^{p_3}} \leq C\|h\chi_{E}\|_{L^{p_1}}\, \|s\chi_{E}\|_{L^{p_2}},\,\, \,h\in L^{p_1},\,s \in L^{p_2}, 
\ee
for every subset $E \subset \mathbb R$ with Lebesgue measure of $E$ is finite. 
Then 
\[
\|fg\|_{p_3,q_3),\theta} \leq D \|f\|_{p_1,q_1),\theta}\, \|g\|_{p_2,q_2),\theta},
\] for  $ f \in  l^{q_1), \theta}(L^{p_1})$ and $ g \in  l^{q_2), \theta}(L^{p_2}),$ 
where $D=c\,C.$
\end{lemma}
\begin{proof} 
Let $ f \in  l^{q_1), \theta}(L^{p_1})$ and $ g \in  l^{q_2), \theta}(L^{p_2}).$ Take $x=\lge\|f\chi_{I_k}\|_{L^{p_1}}\rge_{k \in X}, y=\lge\|g\chi_{I_k}\|_{L^{p_2}}\rge_{k \in X}.$
By the definition  \eqref{eq10}, $x \in l^{q_1), \theta}$  and $y \in l^{q_2), \theta}.$ Applying  \eqref{eq19} successively on the intervals $I_k=[k, k+1),$ with $h=f \chi_{I_k} \in L^{p_1}$ and $s=g \chi_{I_k} \in L^{p_2},$ and then using \eqref{eq18}, we obtain
\begin{align*}
\|fg\|_{p_3, q_3), \theta}&=\big\| \| (fg)\chi_{I_k}\|_{L^{p_3}}\big\|_{l^{q_3), \theta}}\\
& \le \big\| C\| f\chi_{I_k}\|_{L^{p_1}}\| g\chi_{I_k}\|_{L^{p_2}}\big\|_{l^{q_3), \theta}} \\
& = c\,C \big\| \| f\chi_{I_k}\|_{L^{p_1}}\big\|_{l^{q_1),\theta}}\big\| \| f\chi_{I_k}\|_{L^{p_2}}\big\|_{l^{q_2), \theta}}\\
& =D\|f\|_{\left.p_1, q_1\right), \theta} \,\|g\|_{\left.p_2, q_2\right), \theta},
\end{align*}
where $D=c\,C.$ Hence the assertion is proved.
\end{proof}

 Before stating the H\"{o}lder's inequality in the aforesaid spaces, we first define the small Lebesgue sequence spaces. Let $\theta>0$ and $1\leq q < \infty.$  The \emph{small Lebesgue sequence space}, denoted by ${l^{q)', \theta}},$ is the space of all sequences $y=\lge y_k\rge_{k \in X}$ such that
\begin{align}\label{a4} 
\|y\|_{l^{q)', \theta}} &:= \inf_{\substack{{|y_k| = \sum_{j \in X} y_{k,j}}\\{y_{k,j}\geq 0}}} \left\{ \sum_{j \in X} \inf_{\varepsilon > 0} \varepsilon^{\frac{-\theta}{q(1+\varepsilon)}} \left( \sum_{k \in X} y_{k,j}^{(q(1+\varepsilon))'} \right)^{\frac{1}{(q(1+\varepsilon))'}} \right\}<\infty,
\end{align}
where $(q(1+\varepsilon))'$ is the conjugate of $q(1+\varepsilon),$ for each $\varepsilon>0.$\\
First we give an example that for a given sequence $y=\lge y_k\rge_{k \in X,}$ such a decomposition exists, that is, for each $k \in X,$ 
each term $|y_k|$ can be expressed as an infinite sum of non-negative terms,
$|y_k| = \sum_{j \in X} y_{k,j}, \, y_{k,j} \geq 0,$ for all $ j.$
Let $y=\lge y_k= \frac{1}{2^k}\rge_{k \in \mathbb N},$ we may take \( y_{k,j} = \frac{1}{2^k} \cdot \frac{1}{2^j} \)
then
\[\sum_{j=1}^{\infty} y_{k,j} = \sum_{j=1}^{\infty} \frac{1}{2^k} \cdot \frac{1}{2^j} 
= \frac{1}{2^k} \sum_{j=1}^{\infty} \frac{1}{2^j} = \frac{1}{2^k} = y_k.\]
Also, a trivial decomposition of a non-negative sequence  $y=\lge y_k\rge_{k \in X}$ always exists as:
$$y_{k,j} = 
\begin{cases}
y_k , & j = 1 \\
0, & j \neq 1.
\end{cases}$$
Now we give an example of a sequence $y=\lge y_k\rge_{k \in X} \in {l^{q)', \theta}}.$\\
{\bf Example 2.}
Let  $m$ be a positive integer. $y= \lge y_k\rge_{k \in \mathbb Z},$ $y_k=$
$\begin{cases}
1 , & |k| \le m \\
0, & |k|>m.
\end{cases}$
Take 
 \( s_{k,j} = y_{k}\frac{1}{2^j}, \) where \( j \in \mathbb N. \) Clearly,$\lge s_{k,j}\rge_{j \in \mathbb N}$ is a decomposition of the sequence $y= \lge y_k\rge_{k \in \mathbb N}.$
Consider,
\begin{align*}
\|y\|_{l^{q)', \theta}} &=\inf_{\substack{{|y_k| = \sum_{j \in \mathbb N} y_{k,j},}\\{y_{k,j}\geq 0}}} \left\{ \sum_{j \in \mathbb N} \inf_{\varepsilon > 0} \varepsilon^{\frac{-\theta}{q(1+\varepsilon)}} \left( \sum_{k \in \mathbb Z } y_{k,j}^{(q(1+\varepsilon))'} \right)^{\frac{1}{(q(1+\varepsilon))'}} \right\}\\
& \leq \left\{ \sum_{j=1}^{\infty} \inf_{\varepsilon > 0} \varepsilon^{\frac{-\theta}{q1+\varepsilon)}} \left( \sum_{k=-\infty}^{\infty} s_{k,j}^{(q(1+\varepsilon))'} \right)^{\frac{1}{(q(1+\varepsilon))'}} \right\}\\
&=\sum_{j=1}^{\infty} \inf_{\varepsilon > 0}  \varepsilon^{\frac{-\theta}{q(1+\varepsilon)}} \frac{1}{2^j} \left( 2m+1 \right)^{\frac{1}{(q(1+\varepsilon))'}}\\
&\le (2m+1)\sum_{j=1}^{\infty} \inf_{\varepsilon > 0} \varepsilon^{\frac{-\theta}{q(1+\varepsilon)}} \frac{1}{2^j}\\
&=(2m+1)\inf_{\varepsilon > 0} \lr \varepsilon^{\frac{-1}{1+\varepsilon}} \rr^{\frac{\theta}{q}} \sum_{j=1}^{\infty} \frac{1}{2^j}= (2m+1)\lr\frac{1}{\mathrm{e} W(1 / \mathrm{e})}\rr^{\frac{\theta}{q}}<\infty.
\end{align*}
Therefore $y \in {l}^{q)', \theta}.$

 Now we prove the lattice property in ${l}^{q)', \theta}.$ To prove it, first we need the following lemma.

\begin{lemma}\label{l1}
Let  $\{ x_k \}_{k \in \mathbb Z}, \,\{ y_k \}_{k \in \mathbb Z} $ be sequences of non-negative real numbers such that $0 \le y_k \leq x_k,$  for all $ k \in \mathbb Z.$ Suppose  $x_k = \sum_{j \in \mathbb Z} x_{k,j}$ is  any non-negative decomposition of $\{ x_k \}_{k \in \mathbb Z}.$  Define a sequence 
\be\label{35}
z_{k,j} = 
\begin{cases} 
x_{k,j} - \max\left( y_k - \sum\limits_{i=-\infty}^{j-1} x_{k,i},\ 0 \right), & 
    \text{if } \sum\limits_{i=-\infty}^{j} x_{k,i} > y_k \\ 
0, & \text{otherwise}.
\end{cases}
\ee
 Then the following hold:
\begin{enumerate}[(a)]
\item  $0 \leq z_{k,j} \leq x_{k,j},\quad \text{for all} \quad k,\, j \in \mathbb Z.$

\item $y_k = \sum_{j\in \mathbb Z} (x_{k,j} - z_{k,j}), \quad\text{for all} \quad k,\, j \in \mathbb Z.$
\end{enumerate}
\end{lemma}

\begin{proof}
Case (i) Let $y_k = x_k$, for all  $ k \in \mathbb Z.$ Then  \(z_{k,j} = 0,\) for all $ k, j \in \mathbb Z.$ Hence  (a) and (b) hold trivially.\\
 Case(ii) Let  \(y_k < x_k,\) for $ k \in \mathbb Z.$ Define $ \hat{j}_k \in \mathbb Z,$ by 
    \[
    \hat{j}_k = \min \left\{ j \in \mathbb Z : \sum_{i=-\infty}^{j} x_{k,i} > y_k \right\}.
    \]
If $ j < \hat{j}_k$ then  $\displaystyle\sum_{i=-\infty}^{j} x_{k,i} \leq y_k$ and by \eqref{35}, we have $ z_{k,j} = 0.$\\
If $j = \hat{j}_k$ then $\displaystyle\sum_{i=-\infty}^{\hat{j}_k-1} x_{k,i} \leq y_k$  and $\sum_{i=-\infty}^{\hat{j}_k} x_{k,i} > y_k.$ Therefore by \eqref{35}, we have
        \[
        z_{k,\hat{j}_k} = x_{k,\hat{j}_k} - \left( y_k - \sum_{i=-\infty}^{\hat{j}_k-1} x_{k,i} \right)
        .\]\\
 If \(j > \hat{j}_k\) then by \eqref{35}, we have\(z_{k,j} = x_{k,j}.\)\\
It follows from all these three possibilities that (a) is satisfied. To verify (b), consider
\begin{align*}
\sum_{j \in \mathbb Z} (x_{k,j} - z_{k,j}) &= \sum_{j<\hat{j}_k} (x_{k,j} - 0) + (x_{k,\hat{j}_k} - z_{k,\hat{j}_k}) + \sum_{j>\hat{j}_k} (x_{k,j} - x_{k,j}) \\
        &= \sum_{j<\hat{j}_k} x_{k,j} + \left( y_k - \sum_{i=-\infty}^{\hat{j}_k-1} x_{k,i} \right) \\
        &= y_k.
\end{align*}
Hence (b) is satisfied. 
 \end{proof}

\begin{remark}
In \cite{fi}, Fiorenza has proved the continuous version of Lemma \ref{l1}. This lemma asserts that if there is a non-negative decomposition of a given non-negative sequence $\{ x_k \}_{k \in \mathbb Z}$ and  if another sequence $\{ y_k \}_{k \in \mathbb Z}$ satisfies $0 \le y_k \leq x_k,$ for all $ k \in \mathbb Z,$ then $\{ y_k \}_{k \in \mathbb Z}$ also admits a non-negative decomposition. Moreover, the lemma gives an explicit construction of such a decomposition.
\end{remark}  

\begin{theorem}\label{th9}
Lattice property holds in the Small Lebesgue sequence space, i.e., if $ x=\{ x_k \}_{k \in \mathbb Z}, \, y=\{ y_k \}_{k \in \mathbb Z} \in l^{q)', \theta} $ such that $0 \le y \leq x,$ i.e., $y_k \leq x_k,$ for all $ k \in \mathbb Z$ then $\|y\|_{l^{q)', \theta}} \leq \|x\|_{l^{q)', \theta}}.$
\end{theorem}

\begin{proof} In view of Lemma \ref{l1}, we have
\begin{align*}
\|x\|_{l^{q)', \theta}}=& \inf_{\substack{{x_k = \sum_{j \in \mathbb Z} x_{k,j}}\\{x_{k,j}\geq 0}}} \left\{ \sum_{j \in \mathbb Z} \inf_{\varepsilon > 0} \varepsilon^{\frac{-\theta}{q(1+\varepsilon)}} \left( \sum_{k \in \mathbb Z} x_{k,j}^{(q(1+\varepsilon))'} \right)^{\frac{1}{(q(1+\varepsilon))'}} \right\}\\
& \ge \inf_{\substack{{x_k = \sum_{j \in \mathbb Z} x_{k,j}}\\{x_{k,j}\geq 0}}} \left\{ \sum_{j \in \mathbb Z} \inf_{\varepsilon > 0} \varepsilon^{\frac{-\theta}{q(1+\varepsilon)}} \left( \sum_{k \in \mathbb Z} (x_{k,j}-z_{k,j})^{(q(1+\varepsilon))'} \right)^{\frac{1}{(q(1+\varepsilon))'}} \right\}\\
& \ge \inf_{\substack{{y_k = \sum_{j \in \mathbb Z} y_{k,j}}\\{y_{k,j}\geq 0}}} \left\{ \sum_{j \in \mathbb Z} \inf_{\varepsilon > 0} \varepsilon^{\frac{-\theta}{q(1+\varepsilon)}} \left( \sum_{k \in \mathbb Z} y_{k,j}^{(q(1+\varepsilon))'} \right)^{\frac{1}{(q(1+\varepsilon))'}} \right\}\\
&=\|y\|_{l^{q)', \theta}}.
\end{align*}
Hence done.
\end{proof}
Next, we prove that the small Lebesgue sequence space is  complete.

\begin{theorem}\label{t2}
For $\theta>0$ and $1\leq q< \infty$, the space ${l^{q)', \theta}}$ is a Banach space.
\begin{proof}
It is sufficient to prove that  for any non negative sequence  $\{y^{(n)}\}_{ n \in \mathbb{N}}$ which is absolutely summable in ${l^{q)',\theta}}$  is summable in ${l^{q)',\theta}}.$
That is, 
\be
\nonumber\Bigg\| \sum_{n=1}^\infty y^{(n)}\Bigg\|_{l^{q)',\theta}} \leq \sum_{n=1}^\infty\|y^{(n)}\|_{l^{q)',\theta}}.
\ee
Since  $\{y^{(n)}\}_{ n \in \mathbb{N}}$ is absolutely summable in ${l^{q)',\theta}},$ we have
\[
\sum_{n=1}^\infty\|y^{(n)}\|_{l^{q)',\theta}}<\infty.
\]
By the definition of the norm of small Lebesgue sequence space, for given  $\delta>0$ there exists a non negative decomposition $\{y_{k,j}^{(n)}\}_{ k \in X , j \in X } $  of $y^{(n)}=\{y_k^{(n)}\}_{k \in X}$ such that

$y_k^{(n)}=\sum_{j \in X}y_{k,j}^{(n)}$
and 
\be
\nonumber
  \sum_{j\in X}\inf_{\varepsilon > 0} \varepsilon^{\frac{-\theta}{q(1+\varepsilon)}} \left( \sum_{k \in X} (y_{k,j}^{(n)})^{(q(1+\varepsilon))'} \right)^{\frac{1}{(q(1+\varepsilon))'}}  < \| y^{(n)}\|_{l^{q)', \theta}}  + \frac{\delta}{2^n}.
\ee
Then for any $m \in \mathbb N,$
\begin{align*}
\Bigg\| \sum_{n=1}^m y^{(n)}\Bigg\|_{l^{q)',\theta}} 
& \le\inf_{\substack{{\sum_{n=1}^m y^{(n)}_k = \sum_{j \in X} \sum_{n=1}^m s^{(n)}_{k,j}}\\{s^{(n)}_{k,j}\geq 0}}} \left\{ \sum_{j \in X} \inf_{\varepsilon > 0} \varepsilon^{\frac{-\theta}{q(1+\varepsilon)}} \left( \sum_{k \in X} \lr\sum_{n=1}^m s^{(n)}_{k,j} \rr^{(q(1+\varepsilon))'} \right)^{\frac{1}{(q(1+\varepsilon))'}} \right\}\\
&\le\inf_{\substack{{\sum_{n=1}^m y^{(n)}_k = \sum_{j \in X} \sum_{n=1}^m s^{(n)}_{k,j}}\\{s^{(n)}_{k,j}\geq 0}}} \left\{ \sum_{j \in X} \inf_{\varepsilon > 0} \varepsilon^{\frac{-\theta}{q(1+\varepsilon)}}  \sum_{n=1}^m \left( \sum_{k \in X} \lr s^{(n)}_{k,j}\rr ^{(q(1+\varepsilon))'} \right)^{\frac{1}{(q(1+\varepsilon))'}} \right\}\\
&\le\sum_{n=1}^m\sum_{j \in X} \inf_{\varepsilon > 0} \varepsilon^{\frac{-\theta}{q(1+\varepsilon)}} \left(   \sum_{k \in X} \lr y^{(n)}_{k,j}\rr ^{(q(1+\varepsilon))'} \right)^{\frac{1}{(q(1+\varepsilon))'}}\\
&< \sum_{n=1}^m \lr\| y^{(n)}\|_{l^{q)', \theta}}  + \frac{\delta}{2^n}\rr\\
&\le \sum_{n=1}^\infty \|y_k^{(n)}\|_{l^{q)',\theta}} + \delta, ~~~~ \text{for all} \,\, \delta>0 ,\,\, m \in \mathbb N.
\end{align*}
Therefore, we have
\be
\nonumber\Bigg\| \sum_{n=1}^m y^{(n)}\Bigg\|_{l^{q)',\theta}} \leq \sum_{n=1}^\infty\|y^{(n)}\|_{l^{q)',\theta}}, ~~~~ \text{for all} \,\, m \in \mathbb N.
\ee
Hence we have proved the assertion.
\end{proof}
\end{theorem}
\begin{remark}
Following the similar steps, the triangle inequality can be proved easily in the small Lebesgue sequence space.
\end{remark}
\begin{theorem}
Let $1 \le q <\infty$ and $\theta >0$ be given. Suppose $x=\lge x_k\rge_{k \in X} \in l^{q),{\theta}}$and  $y=\lge y_k\rge_{k \in X} \in  l^{q)',{\theta}}$then
\be\label{a3}
\sum_{k\in X} |x_k y_k |\leq \|x\|_{l^{q), \theta}} \|y\|_{l^{q)', \theta}}.
\ee
\end{theorem}
\begin{proof}Let $x=\lge x_k\rge_{k \in X} \in l^{q),{\theta}}$ and  $y=\lge y_k\rge_{k \in X} \in  l^{q)',{\theta}}.$ Suppose 
$|y_k| = \sum_{j\in X} y_{k,j}, \,\, y_{k,j} \geq 0,$ is a 
decomposition of $y=\lge y_k\rge_{k \in X}.$ For $\varepsilon>0,$ applying H\"{o}lder's inequality for the exponents $q(1+\varepsilon)$ and  $(q(1+\varepsilon))'$ in the classical Lebesgue sequence spaces,  we get
\begin{align}\label{a1}
\nonumber\sum_{k \in X} |x_k| y_{k,j} &\leq \left( \sum_{k \in X} |x_k|^{q(1+\varepsilon)} \right)^{\frac{1}{q(1+\varepsilon)}} \left( \sum_{k \in X} y_{k,j}^{(q(1+\varepsilon))^{\prime}} \right)^{\frac{1}{(q(1+\varepsilon))^{\prime}}}\\\nonumber
&=\varepsilon^{\frac{\theta}{q(1+\varepsilon)}}\left( \sum_{k \in X} |x_k|^{q(1+\varepsilon)} \right)^{\frac{1}{q(1+\varepsilon)}}\varepsilon^{\frac{-\theta}{q(1+\varepsilon)}}  \left( \sum_{k \in X} y_{k,j}^{(q(1+\varepsilon))^{\prime}} \right)^{\frac{1}{(q(1+\varepsilon))^{\prime}}}\\\nonumber
&  \leq  \|x\|_{l^{q),{\theta}}} \,\,\varepsilon^{\frac{-\theta}{q(1+\varepsilon)}}  \left( \sum_{k \in X} y_{k,j}^{(q(1+\varepsilon))^{\prime}} \right)^{\frac{1}{(q(1+\varepsilon))^{\prime}}}\\
&  \leq  \|x\|_{l^{q),{\theta}}} \,\,\inf_{\varepsilon>0}\varepsilon^{\frac{-\theta}{q(1+\varepsilon)}}  \left( \sum_{k \in X} y_{k,j}^{(q(1+\varepsilon))^{\prime}} \right)^{\frac{1}{(q(1+\varepsilon))^{\prime}}}.
\end{align}
Now, by using \eqref{a1}, we obtain
\begin{align*}
\sum_{k\in X} |x_k y_k |&=\sum_{k\in X} |x_k|\lr\sum_{j\in X} y_{k,j}\rr=
\sum_{j\in X}\sum_{k\in X}|x_k|y_{k,j}\\
& \le \sum_{j\in X} \|x\|_{l^{q),{\theta}}} \,\,\inf_{\varepsilon>0}\varepsilon^{\frac{-\theta}{q(1+\varepsilon)}}  \left( \sum_{k \in X} y_{k,j}^{(q(1+\varepsilon))^{\prime}} \right)^{\frac{1}{(q(1+\varepsilon))^{\prime}}}\\
&\le\|x\|_{l^{q),{\theta}}}\inf_{\substack{{|y_k| = \sum_{j \in X} y_{k,j}}\\{y_{k,j}\geq 0}}}\left\{\sum_{j\in X} \inf_{\varepsilon>0}\varepsilon^{\frac{-\theta}{q(1+\varepsilon)}}  \left( \sum_{k \in X} y_{k,j}^{(q(1+\varepsilon))^{\prime}} \right)^{\frac{1}{(q(1+\varepsilon))^{\prime}}}\right\}\\
&=\|x\|_{l^{q), \theta}} \|y\|_{l^{q)', \theta}}.
\end{align*}
Hence proved.
\end{proof}

 Let  $1 \leq p < \infty,$  $1 \leq q < \infty$, and $\theta > 0$ then the \emph{auxiliary space of the grand amalgam Lebesgue function space,} denoted by  $l^{q)', \theta}(L^p),$ is the collection of all complex valued measurable functions $f$  defined on $\displaystyle\bigcup_{k\in X} I_k,$ where $I_k=[k, k+1),$ with $ f\chi_{I_k} \in L^p,$ for each $ k \in X,$ such that
\begin{align*} 
\|f\|_{p, q)', \theta} &:= \inf_{\substack{{\| f\chi_{I_k} \|_{L^p} = \sum_{j \in X} y_{k,j}}\\{y_{k,j}\geq 0}}} \left\{ \sum_{j \in X} \inf_{\varepsilon > 0} \varepsilon^{\frac{-\theta}{q(1+\varepsilon)}} \left( \sum_{k \in X} y_{k,j}^{(q(1+\varepsilon))'} \right)^{\frac{1}{(q(1+\varepsilon))'}} \right\}<\infty,
\end{align*}
where $\frac{1}{q(1+\varepsilon)}+\frac{1}{(q(1+\varepsilon))'}=1,$ for every $\varepsilon>0.$ In view of \eqref{a4} this norm can also be written as:
\be
\nonumber \|f\|_{p, q)', \theta}=\big\|\|f\chi_{I_k}\|_{L^p}\big\|_{l^{q)', \theta}}.
\ee 
In view of the Theorem 1 of \cite{ho} and Theorem \ref{t2}, auxiliary space of the grand amalgam Lebesgue function space is a Banach space.\\
Now, we prove the \noindent\emph{ H\"{o}lder's inequality on grand amalgam  Lebesgue function space}.

\begin{theorem}
Let $ \theta>0,\,\, X= \mathbb Z$ and $1 \le p, \, q <\infty.$ If \( g\in l^{q), \theta}(L^p) \) and \( f \in l^{{q})',\theta}(L^{p'}) \) then \text{$gf \in L^1 \text{ and }$}
\be\label{a6}
\int_{\mathbb R} |g(x)f(x)| dx \leq \|g\|_{p,q),\theta}\, \|f\|_{p',q)',\theta},
\ee
where $\frac{1}{p}+\frac{1}{p'}=1.$
\end{theorem}
\begin{proof}
Let \( g \in l^{q),\theta}(L^p) \)  and \( f \in l^{q)',\theta}(L^{p'}). \) Using the H\"{o}lder's inequality for the exponents $p$ and  $p'$ in the  Lebesgue spaces, we get
\begin{equation}\label{a2}
\int_{I_k} |g(x)f(x)| dx \leq \| g \chi_{I_k}\|_{L^p} \|f \chi_{I_k}\|_{L^{p'}}, \,\, \text{for} \, \, k \in X.
\end{equation}
Using \eqref{a2} and  H\"{o}lder's inequality \eqref{a3} in the grand  Lebesgue sequence spaces, we obtain 
\begin{align*}
\int_{\mathbb R} |g(x)f(x)| dx&=\sum_{k \in X} \int_{I_k} |g(x)f(x)| dx \\
& \leq \sum_{k \in X} \| g \chi_{I_k}\|_{L^p} \,\|f \chi_{I_k}\|_{L^{p'}}\\
&\leq \big\|\| g \chi_{I_k}\|_{L^p}\big\|_{l^{q), \theta}}\, \left\|  \|f\chi_{I_k}\|_{L^{p'}} \right\|_{l^{q)', \theta}}\\
&=\|g\|_{p,q),\theta} \|f\|_{p',q)',\theta}. 
\end{align*}
Hence done.
\end{proof}

\begin{remark}
Since the H\"{o}lder's inequality holds in both  the grand  Lebesgue sequence spaces $l^{q), \theta}$ and the classical Lebesgue space $L^p,$ it also holds in grand amalgam Lebesgue function space $l^{q), \theta}(L^p).$ This is the general fact concerning the amalgamation of two normed spaces, see \cite{fs}.
\end{remark}

In the next theorem, we show that the characteristic function  $\chi_E$ on bounded measurable subset $ E \subset \mathbb R$ belongs to $l^{q),\theta}(L^p)$ and $l^{q)',\theta}(L^p)$ as well. Also, we show that the condition of boundedness is necessary.

\begin{theorem}\label{th10} 
Let E be  any bounded subset of $\mathbb R.$ Then 
\begin{enumerate}[(a)]
\item $\chi_ E  \in l^{q),\theta}(L^p).$ 
\item $\chi_ E  \in l^{q)',\theta}(L^p).$
\item there exists a constant $c_E>0$ such that 
\be\label{eq21} 
\displaystyle\int_E|g(x)| d x \leq c_E\|g\|_{p, q), \theta},
\ee
for every measurable function $g \in l^{q),\theta}(L^p).$
\end{enumerate}
\end{theorem}
\begin{proof}
\begin{enumerate}[(a)]
\item Let  $ E $ be any bounded subset of $\mathbb R$. This implies that there exists positive integer $M$ such that $ E \subset [-M, M].$ Therefore, 
 $ E \cap [k, k+1) = \emptyset,$ for every $k$ such that $ k <-M $ and $k \ge M. $ Hence Lebesgue measure 
$|E \cap I_k|$ is non-zero for at the most $2M$ intervals $I_k$.
\begin{align*}
\|\chi_E\|_{p,q),\theta} &= \sup_{\varepsilon>0 } \lr {\varepsilon^{\theta}}  \sum_{k \in  X} \left(\int_{k}^{k+1} |\chi_E|^{p} \, dx \right)^ \frac{q(1+\varepsilon)}{p}\rr^\frac{1}{{q(1+\varepsilon)}} \\
&=\sup_{\varepsilon>0 }\lr {\varepsilon^{\theta}} \sum_{k=-M} ^{M-1}\left(|E \cap I_k|\right)^{\frac{q(1+\varepsilon)}{p}}\rr^\frac{1}{{q(1+\varepsilon)}}\\
& \le \sup_{\varepsilon>0 }\lr {\varepsilon^{\theta}} 2M\rr^\frac{1}{{q(1+\varepsilon)}}
\le \lr2M\rr^{\frac{1}{q}} \lr\frac{1}{\mathrm{e} W(1 / \mathrm{e})}\rr^{\frac{\theta}{q}},
\end{align*}
where $|E \cap I_k|$ denotes Lebesgue measure of $E \cap I_k.$
Hence $ ||\chi_E||_{p, q), \theta} < \infty.$\\
\item Take $y= \lge y_k\rge_{k \in \mathbb Z},$ $y_k=$
$\begin{cases}
1 , & -M \le k <M-1 \\
0, & \text{otherwise}.
\end{cases}$\\
In view of Example 2, $y \in l^{q)',\theta}.$
Clearly, $\lge\|\chi_{E\cap I_k} \|_{L^p}\rge_{k \in \mathbb Z}=|{E\cap I_k}|_{k \in \mathbb Z}^{\frac{1}{p}} \le \lge y_k\rge_{k \in \mathbb Z}=y,$ see argument given in the above proof. By using the  Theorem \ref{th9} (lattice property of the small Lebesgue sequence space), we have $\big\|\|\chi_{E\cap I_k} \|_{L^p}\big\|_{l^{q)', \theta}}
\le \big\|y \big\|_{l^{q)', \theta}}.$\\
 Now consider,
\begin{align*}
\|\chi_E\|_{p,q)',\theta} &=\inf_{\substack{{\| \chi_E\chi_{I_k} \|_{L^p} = \sum_{j \in X} y_{k,j}}\\{y_{k,j}\geq 0}}} \left\{ \sum_{j \in X} \inf_{\varepsilon > 0} \varepsilon^{\frac{-\theta}{q(1+\varepsilon)}} \left( \sum_{k \in X} y_{k,j}^{(q(1+\varepsilon))'} \right)^{\frac{1}{(q(1+\varepsilon))'}} \right\}\\
&=\inf_{\substack{{\| \chi_{E\cap I_k} \|_{L^p} = \sum_{j \in X} y_{k,j}}\\{y_{k,j}\geq 0}}} \left\{ \sum_{j \in X} \inf_{\varepsilon > 0} \varepsilon^{\frac{-\theta}{q(1+\varepsilon)}} \left( \sum_{k \in X} y_{k,j}^{(q(1+\varepsilon))'} \right)^{\frac{1}{(q(1+\varepsilon))'}} \right\}\\
&=\big\|\|\chi_{E\cap I_k} \|_{L^p}\big\|_{l^{q)', \theta}}\\
&\le \big\|y \big\|_{l^{q)', \theta}},
\end{align*}
which is finite. Hence done.
 \item By using the  H\"{o}lder inequality \eqref{a6} in the grand amalgam Lebesgue function spaces, we get
\be
\nonumber \displaystyle\int_E|g(x)| d x=\int_{\mathbb R}\left|g \chi_E\right| d x
\le\|g\|_{p,q),\theta} \|\chi_E\|_{p',q)',\theta}=c_E\|g\|_{p, q), \theta},
\ee
where $c_E=\|\chi_E\|_{p',q)',\theta},$ which is finite with the preceding proof. Hence proved.
\end{enumerate}
\end{proof}
\begin{remark}
In Theorem \ref{th10}, the condition of boundedness on $E \subset \mathbb R $ is  necessary. We give an example that if   the condition of the boundedness is relaxed than it is not necessary that  $\chi_E \in l^{q), \theta}(L^p).$
Choose $ 1 \leq q < p < \infty $, $\theta =\frac{1}{2}.$ Take $\alpha = \frac{p}{q} > 1,$ and $E = \bigcup_{n \in \mathbb{N}} \left[n, n + \frac{1}{n^\alpha}\right].$ Clearly $E$ is not bounded but
Lebesgue measure $|E| = \sum_{n=1}^{\infty} n^{-\alpha} $ is finite because $\alpha>1.$ Now, 
\begin{align*}
\|\chi_E\|_{p,q),\theta} &= \sup_{\varepsilon > 0} \lr {\varepsilon^{\theta}}  \sum_{n \in \mathbb N} \left(\int_{n}^{n+1} |\chi_E|^{p} \, dx \right)^ \frac{q(1+\varepsilon)}{p}\rr^\frac{1}{{q(1+\varepsilon)}}\\
 &=\sup_{\varepsilon > 0} \lr {\varepsilon^{\theta}}  \sum_{n \in \mathbb N} \left(\int_{E\cap I_n} 1\, dx \right)^ \frac{q(1+\varepsilon)}{p}\rr^\frac{1}{{q(1+\varepsilon)}}\\ 
&=\sup_{\varepsilon > 0} \lr {\varepsilon^{\theta}}  \sum_{n \in \mathbb N} \left(|E\cap I_n|^{\frac{1}{p}}\right)^{q(1+\varepsilon)}\rr^\frac{1}{{q(1+\varepsilon)}}\\
&= \left\||E\cap I_n|^{\frac{1}{p}}\right\|_{l^{q),\theta}}\\
&=\left\|n^{\frac{-\alpha}{p}}\right\|_{l^{q),\theta}}=\|n^{-\frac{1}{q}}\|_{l^{q),\theta}}.
\end{align*}
Since $\theta<1,$ by Example 1, $\|\chi_E\|_{p,q),\theta}=\infty.$ Therefore, $\chi_E \notin l^{q), \theta}(L^p).$

\vspace{3pt}

Now we give an example that \eqref{eq21} also may or may not hold if the condition of boundedness of $E \subset \mathbb R$ is relaxed.
Let $1 \leq p < q < \infty, \, \,\theta =2.$ Take  \(\alpha = \frac{p}{q}<1\),\(\,\beta = \frac{q - \alpha}{q - 1}>1,\) $E = \bigcup_{n \in \mathbb{N}} \left[n, n + \frac{1}{n^\beta}\right].$ Clearly $E$ is not bounded but $|E| = \sum_{n=1}^{\infty} n^{-\beta}$ is finite because $\beta >1.$ Let $g = \sum_{n=1}^{\infty} n^{\frac{\beta - \alpha}{p}} \chi_{\left[n, n + \frac{1}{n^\beta}\right]}.$
\begin{align*}
\|g\|_{p,q),\theta} = & \sup_{\varepsilon > 0} \left\{ {\varepsilon^{\theta}}  \sum_{n \in \mathbb{N}} \left(\int_{E \cap I_n} |g(x)|^{p} \, dx \right)^ \frac{q(1+\varepsilon)}{p}\right\}^\frac{1}{{q(1+\varepsilon)}} \\
& = \sup_{\varepsilon > 0} \left\{ {\varepsilon^{\theta}}  \sum_{n \in \mathbb{N}} \left(n^{{\beta - \alpha}}  n^{-\beta}\right)^ \frac{q(1+\varepsilon)}{p}\right\}^\frac{1}{{q(1+\varepsilon)}}\\
&= \sup_{\varepsilon > 0} \left\{ {\varepsilon^{\theta}}  \sum_{n \in \mathbb{N}} \left(n^{\frac{-\alpha}{p}} \right)^ {q(1+\varepsilon)}\right\}^\frac{1}{{q(1+\varepsilon)}}\\
&=\|n^{-\frac{\alpha}{p}}\|_{l^{q),\theta}}=\|n^{-\frac{1}{q}}\|_{l^{q),\theta}}
\end{align*}
Since $\theta>1,$ by Example 1, $\|g\|_{p,q),\theta}<\infty.$ 
\end{remark}
\noindent But 
\[\int_{E} |g| dx = \sum_{n=1}^{\infty} n^{\frac{\beta - \alpha}{p} {-\beta}}  =\sum_{n=1}^{\infty} \lr\frac{1}{n}\rr^{{\beta}-\frac{\beta - \alpha}{p}}=\infty,\]
 because ${{\beta}-\frac{\beta - \alpha}{p}}<1.$ 

\begin{remark}
From Theorem \ref{th10}, it follows that the grand amalgam Lebesgue function space is not a Banach function space, for definition, see \cite{bs}.
\end{remark}

\section{Multiplication operator on $l^{q),\theta}(L^{p})$}

In this section, we study the multiplication operator in the frame work of    $l ^{q),\theta}(L^p)$ space. Let $g: \mathbb R \rightarrow \mathbb R$ be a measurable function. A \emph{Multiplication operator} $M_g$ is defined as $M_g(f)= fg,$ where  $f$ 
is a measurable function defined on $\mathbb R.$

\begin{theorem}
Let \( 1 \leq p,q < \infty \), $X=\mathbb Z$ and  $\theta > 0.$ Suppose that $g$is a measurable function with compact support. Then  the multiplication operator 
\be
 \nonumber M_g : l^{q),\theta}(L^p) \to l^{q),\theta}(L^p) \text{ is bounded if and only if } g \in L^\infty.
\ee
Moreover,
\[
\|M_g\| = \|g\|_{\infty}.
\]
\begin{proof}
 Assume that \( g \in L^\infty \). For $f \in l^{q),\theta}(L^p),$ consider
\begin{align*}
\|M_g f\|_{p,q),\theta} &= \sup_{\varepsilon > 0} \lr {\varepsilon^{\theta}}  \sum_{n \in  \mathbb Z} \left(\int_{n}^{n+1} |M_g(x)|^{p} \, dx \right)^ \frac{q(1+\varepsilon)}{p}\rr^\frac{1}{{q(1+\varepsilon)}} \\
&= \sup_{\varepsilon > 0} \lr {\varepsilon^{\theta}}  \sum_{n \in  \mathbb Z} \left(\int_{n}^{n+1} |(fg)(x)|^{p} \, dx \right)^ \frac{q(1+\varepsilon)}{p}\rr^\frac{1}{{q(1+\varepsilon)}} \\
&\le  \|g\|_{L^\infty}\sup_{\varepsilon > 0} \lr {\varepsilon^{\theta}}  \sum_{n \in  \mathbb Z} \left(\int_{n}^{n+1} |f(x)|^{p} \, dx \right)^ \frac{q(1+\varepsilon)}{p}\rr^\frac{1}{{q(1+\varepsilon)}} \\
&=\|g\|_{L^\infty} \|f\|_{p,q),\theta}. 
\end{align*}
Also, the operator norm of $M_g$ is
\begin{equation}\label{b1}
\|M_g\| = \sup_{f \neq 0} \frac{\|M_g f\|_{p,q),\theta}}{\|f\|_{p,q),\theta}} \leq \|g\|_{\infty}.
\end{equation}
 \noindent Conversely, assume that \( M_g \) is bounded but \( g \notin L^\infty \). Then for every \( m \in \mathbb{N} \), the set
\[ G_m = \{ x \in \mathbb{R} : |g(x)| > m \} \]
has positive measure. Since $g $ has compact support, this implies that there exists a bounded subset \( K_m \) of \( G_m \) with positive measure. Hence by Theorem \ref{th10},  \( \|\chi_{K_m}\|_{p,q),\theta}\) is finite.
Observe that
\be
\nonumber |(g \chi_{K_m})(x)| \geq m \chi_{K_m}(x),  \quad \text{for any} \,\,\,  m \in \mathbb{N}. 
\ee 
It follows that
\begin{align*}
\|g \chi_{K_m}\|_{p,q),\theta}&= \sup_{\varepsilon > 0} \lr {\varepsilon^{\theta}}  \sum_{n \in  \mathbb Z} \left(\int_{n}^{n+1} |g \chi_{K_m}(x)|^{p} \, dx \right)^ \frac{q(1+\varepsilon)}{p}\rr^\frac{1}{{q(1+\varepsilon)}}\\
& \ge \sup_{\varepsilon > 0} \lr {\varepsilon^{\theta}}  \sum_{n \in  \mathbb Z} \left(\int_{n}^{n+1} |m \chi_{K_m}(x)|^{p} \, dx \right)^ \frac{q(1+\varepsilon)}{p}\rr^\frac{1}{{q(1+\varepsilon)}}\\
&= m \|\chi_{K_m}\|_{p,q),\theta}.
\end{align*}
This implies that
\[ \|M_g \chi_{K_m}\|_{p,q),\theta} \geq m \|\chi_{K_m}\|_{p,q),\theta}, \text{ for each } \,\, m \in \mathbb N.\]
It contradicts the boundedness of \( M_g \) on $l^{q),\theta}(L^p).$ Therefore, \( g \in L^\infty .\)

Now, we shall prove that \( \|M_g\| = \|g\|_{\infty} \).\\ 
For any \( \delta > 0 \), define a set \( G = \{ x \in \mathbb{R} : |g(x)| > \|g\|_{\infty} - \delta\} \). Then \( G \) has positive measure.  Since $g $ has compact support, there exists a bounded subset \( K\) of \( G, \) and by Theorem \ref{th10},  \( \|\chi_{K}\|_{p,q),\theta}\) is finite. Observe that
\be
\nonumber |(g \chi_{K})(x)| \geq  (\|g\|_{\infty} - \delta)   \chi_{K}(x),\,\, \text{for every}\,\, x \in K.
\ee 
It follows that
\begin{align*}
\|g \chi_{K}\|_{p,q),\theta}&= \sup_{\varepsilon > 0} \lr {\varepsilon^{\theta}}  \sum_{n \in  \mathbb Z} \left(\int_{n}^{n+1} |g \chi_{K}(x)|^{p} \, dx \right)^ \frac{q(1+\varepsilon)}{p}\rr^\frac{1}{{q(1+\varepsilon)}}\\
& \ge \sup_{\varepsilon > 0} \lr {\varepsilon^{\theta}}  \sum_{n \in  \mathbb Z} \left(\int_{n}^{n+1} |(\|g\|_{\infty} - \delta) \chi_{K}(x)|^{p} \, dx \right)^ \frac{q(1+\varepsilon)}{p}\rr^\frac{1}{{q(1+\varepsilon)}}\\
&= (\|g\|_{\infty} - \delta) \|\chi_{K}\|_{p,q),\theta}.
\end{align*}
Therefore, 
\be\label{b2}
\|M_g\|  \geq \|g\|_{\infty} - \delta, \,\, \text{for all} \,\,\delta >0
\ee 
Hence, from \eqref{b1} and \eqref{b2}, we have
\[
\|M_g\| = \|g\|_{\infty}.
\]
\end{proof}
\end{theorem}

\begin{theorem}
Let \( 1 \leq p,q < \infty \), $X=\mathbb Z$ and  $\theta > 0.$ If \( g \in l^{q)', \theta}(L^{p'})\), then the multiplication operator 
\( M_g :l^{q), \theta}(L^p)  \to L^1(\mathbb{R}) \) is bounded and
$
\|M_g\| \leq  \|g\|_{p',q)',\theta}
,$ where $\frac{1}{p}+\frac{1}{p'}=1.$
\begin{proof}
To prove the theorem, we shall use the Closed Graph Theorem. First we shall prove that the multiplication operator is a closed operator. Let  $\{ f_n \}_{n \in \mathbb{N}}$ be a sequence in $l^{q), \theta}(L^p)$ and $\{ f_n  \}_{n \in \mathbb{N}} \to f $ in the norm of \(l^{q), \theta}(L^p)\), where \(f \in l^{q), \theta}(L^p).\)
Suppose \( M_g(f_n) = g f_n \to h \) in \( L^1(\mathbb{R}) \). By using Lemma \ref{lm 1}, $\{ f_n \}_{n \in \mathbb{N}}$ has a subsequence $\{ {f_{n_k}} \} $ which converges pointwise almost everywhere (a.e.) to $f$.
Therefore, we obtain  \( g f_{n_k} \to g f \) pointwise a.e. So \( h = g f \) pointwise a.e. This shows that 
the graph of \( M_g \) is closed, hence the multiplication operator $M_g$  is a closed operator.\\
Since $l^{q), \theta}(L^p)$and \( L^1(\mathbb{R}) \) are Banach spaces and the multiplication operator is linear, by using  Closed Graph Theorem, we have the assertion.\\  
Now to estimate the norm of $M_g,$ we shall use the 
 H\"{o}lder's inequality in the grand amalgam Lebesgue function spaces, i.e., $l^{q), \theta}(L^p)$. Consider,
\be
\nonumber\|M_g f\|_{L^1} = \|g f\|_{L^1} \leq  {\|f\|_{p,q),\theta} \|g\|_{p',q)',\theta}}
\ee
\[ 
\|M_g\| = \sup_{f \neq 0} \frac{\|M_g f\|_{L^1}}{\|f\|_{p,q),\theta}} \leq  \|g\|_{p',q)',\theta}.\]
Hence done.
\end{proof}
\end{theorem}
\begin{theorem}
Let \( 1 \leq p,q < \infty \), $X=\mathbb Z$ and  $\theta > 0.$ Suppose that $g$ is a measurable function with compact support and  \( M_g \) is a bounded operator from $l^{q), \theta}(L^p)$ to $l^{q), \theta}(L^p)$. Then \( M_g \) is an isometry if and only if \( |g(x)| = 1 \) a.e.
\begin{proof}
Let $M_g$ be an isometry and \( |g(x)| \neq 1 \) a.e., then there is a measurable subset $C \subset \mathbb{R}$ of positive measure such that $|g(x)| < 1,$ for $ x \in C$ or there exists a measurable subset $D \subset \mathbb R$ of positive measure such that $|g(x)| > 1,$ for  $ x \in D.$

 If $|g(x)| < 1,$ we  can choose a subset  $C_\delta = \{x \in \mathbb{R} : |g(x)| < 1 - \delta\}$ such that $|C_\delta|>0,$ for some $\delta > 0$. Since $g $ has compact support, this implies that there exists a bounded subset \( E \) of \( C_\delta, \) and by Theorem \ref{th10},  \( \|\chi_{E}\|_{p,q),\theta}\) is finite.
Observe that
\be
\nonumber |(g \chi_{E})(x)| < (1 - \delta) \chi_{E}(x),  \quad \text{for some} \,\,\, \delta >0. 
\ee 
It follows that,
\begin{align*} 
\|M_g \chi_E\|_{p,q),\theta}&= \sup_{\varepsilon > 0} \lr {\varepsilon^{\theta}}  \sum_{n \in \mathbb{Z}} \left(\int_{n}^{n+1} |(g \chi_{E})(x)|^{p} \, dx
\right)^ \frac{q(1+\varepsilon)}{p}\rr^\frac{1}{{q(1+\varepsilon)}}\\
&< (1-\delta)\sup_{\varepsilon > 0} \lr {\varepsilon^{\theta}}  \sum_{n \in \mathbb{Z}} \left(\int_{n}^{n+1} | \chi_{E}(x)|^{p} \, dx
\right)^ \frac{q(1+\varepsilon)}{p}\rr^\frac{1}{{q(1+\varepsilon)}}\\
&=(1-\delta)\| \chi_E\|_{p,q),\theta}<\| \chi_E\|_{p,q),\theta}.
\end{align*}
Since  $\delta$ is arbitrary, we have $\|M_g \chi_E\|_{p,q),\theta}< \|\chi_E\|_{p,q),\theta}.$ It is a  contradiction that  $M_g$ is an isometry.
Similarly, for the other case, just by following the similar steps,  we get the contradiction. Hence \( |g(x)| = 1 \) a.e.\\
Conversely, let \( |g(x)| = 1 \) a.e. then
\begin{align*}
\|M_g f\|_{p,q),\theta}=&  \sup_{\varepsilon > 0} \lr {\varepsilon^{\theta}}  \sum_{n \in \mathbb{Z}} \left(\int_{n}^{n+1} |(fg )(x)|^{p} \, dx
\right)^ \frac{q(1+\varepsilon)}{p}\rr^\frac{1}{{q(1+\varepsilon)}}\\
 =& \sup_{\varepsilon > 0} \lr {\varepsilon^{\theta}}  \sum_{n \in \mathbb{Z}} \left(\int_{n}^{n+1} |f (x)|^{p} \, dx
\right)^ \frac{q(1+\varepsilon)}{p}\rr^\frac{1}{{q(1+\varepsilon)}}
=\|f\|_{p,q),\theta}.
\end{align*}
Hence  $M_g$ is an isometry.
\end{proof}
\end{theorem}

\bigskip 
 
\emph{Acknowledgement}: The first author acknowledges the  National Board of Higher Mathematics research project no. 02011/14/2023 NBHM(R.P)/R\&D II/5951, India.\\
Second author acknowledges University Grants Commission (UGC) for UGC - JRF fellowship (NTA Ref.No.:221610184366).

\bigskip

\noindent Monika Singh,
 Department of Mathematics\\
Lady Shri Ram College for Women (University of Delhi)\\
Lajpat Nagar, Delhi - 110024, India\\
(Email : monikasingh@lsr.du.ac.in)

\bigskip
\noindent Jitendra Kumar,
  Department of Mathematics\\
University of Delhi, Delhi - 110007, India\\
(Email: jk5129969@gmail.com)


\begin{thebibliography}{99}\itemsep 3pt

\bibitem{ay} 
I. Aydin, \emph{On a subalgebra of $L_{w}^{1}(G)$}, Stud. Univ. Babes-Bolyai Math. 57(2012), No.4, 527-539.
\bibitem{ak} 
I. Aydin, O. Kulak, \emph{On the properties of multiplication operators in some function spaces}, Montes Taurus J. Pure Appl. Math., (2023), 1-11.

\bibitem{bbc}
A. L. Baison, J. B.Contreras, and V. A. Cruz, \emph{ The maximal operator on
 the Amalgam space,} Zurnal matematiceskoj fiziki analiza geometrii, 19(4)(2023), 679-695.

\bibitem{bs}
C. Bennet and R. Sharpley, \emph{Interpolation of operators}, Academic Press, INC. Orlando, Florida, 1988.


\bibitem{bl}
J. Bergh and J. L\"{o}fstr\"{o}m, \emph{Interpolation spaces: An introduction,} Grundlehren Math. Wiss. 223(1976), Springer, Berlin. 


\bibitem{cghjk} 
R.M. Corless, G.H. Gonnet, D.E.G. Hare, D.J. Jeffrey and D.E. Knuth,\emph{on the Lambert W function}, Adv. Comput. Math., 5(1996), 329-359.

\bibitem{cl} 
C. Carton-Lebrun (Mons), H.P. Heinig (Hamilton,Ont) and S.C. Hofmann (Dayton, Ohio), \emph{Integral operators on weighted amalgams}, Studia Math. 109(2)(1994), 133-157. 

\bibitem{ch} 
H. C. Chaparro, \emph{Multiplication operators between mixed norm Lebesgue spaces}, Hacet. J. Math. Stat. 48 (3)(2019), 771–778.


\bibitem{fi}
A. Fiorenza, \emph{Duality and reflexivity in grand Lebesgue spaces,} Collect. Math. 51(2)(2000), 131-148.


\bibitem{fk}
A. Fiorenza and G.E. Karadzhov, \emph{Grand and small Lebesgue spaces and their analogs, } Z. Anal. Anwend, 23(4)(2004), 657-681.


\bibitem{fs}
J.J.F. Fournier and J. Stewart, \emph{Amalgams of $L^P$ and $l^q$,} Bull. Amer. Math. Soc. (N.S.),  13(1995), 1-21.

\bibitem{he}
C. Heil, \emph{An introduction to weighted Wiener amalgams}, In: Wavelets and their Applications (Chennai, 2002), M. Krishna R. Radha and S. Thangavelu, eds., Allied Publisher, New Delhi, (2003) 183-216.


\bibitem{ho} 
F. Holland, \emph{Harmonic Analysis on Amalgams of $L^p$ and $l^q$}, J. London Math. Soc., 2(10)(1975), 295-305.

\bibitem{is}
T. Iwaniec and C. Sbordone, \emph{On the integrability of the Jacobian under minimal hypotheses}, Arch. Ration. Mech. Anal., 119(1992), 129-143.


\bibitem{jss2}
P. Jain, M. Singh and A. P. Singh, \emph{Recent trends in grand Lebesgue spaces}, In: Function Spaces and Inequalities, Eds. P. Jain et al, Springer,(2017), 137-159.


\bibitem{kmrs}
V. Kokilashvili, A. Meskhi, H. Rafeiro and S. Samko \emph{Integral Operators in Non-Standard Function Spaces:Advances in Grand Function Spaces,} Birkh\"{a}user, 3(2024).

\bibitem{o} 
O. Ogur, \emph{Grand Lorentz sequence space and its Multiplication operator}, Commun. Fac. Sci. Univ. Ank. Ser A1 Math. Stat., 69(1)(2020), 771-781.

\bibitem{rsu}
H. Rafeiro, S. Samko and S. Umarkhadzhiev, \emph{Grand Lebesgue sequence spaces}, Georgian Math.J. 25(2)(2018), 291-302.


\bibitem{t1} 
A. Turan Gurkanli, \emph{On the grand Wiener amalgam space}, Rocky Mountain J. Math. 50(5)(2020), 1647-1659.

\bibitem{t2} 
A. Turan Gurkanli, \emph{New properties of grand amalgam spaces}, DOI: 10.48550/arXiv.1901.07282, 2019. 


\bibitem{w1}
N. Wiener, \emph{On the representation of functions by trigonometrical integrals}, Math. Z., 24(1926), 575-616.

\bibitem{w2}
N. Wiener, \emph{Tauberian theorems}, Ann. of Math., 33(1932), 1-100.
\end{thebibliography}
\end{document}